\documentclass[leqno,12pt]{article}
\usepackage[pagewise,mathlines]{lineno}
\usepackage{graphicx}
\usepackage{amssymb}
\usepackage{CJK}
\usepackage[utf8]{inputenc}
\usepackage{amssymb}
\usepackage{amsthm}
\usepackage{amsfonts}
\usepackage{amsmath}
\usepackage[normalem]{ulem}
\usepackage{tikz}
\usepackage{hyperref}
\usepackage[]{pstricks}
\usepackage{mathrsfs}
\usepackage[]{epsfig}
\usepackage{indentfirst}
\everymath{\displaystyle}
\numberwithin{equation}{section}
\newtheorem{remark}{Remark}

\usepackage{amsfonts}
\usepackage[]{amsmath}
\usepackage[]{float}

\usepackage[]{epsfig}
\usepackage[]{pstricks}
\usepackage{tikz}

\newtheorem{theorem}{Theorem}[section]
\newtheorem{lemma}{Lemma}[section]

\numberwithin{equation}{section}
\newcommand{\R}{\mathbb R}

\begin{document}

\title{The nonlinear Quadratic Interactions of the Schr\"odinger type on the half-line}

\author{Isnaldo Isaac Barbosa and Márcio Cavalcante\\ \footnotesize{Instituto de Matemática}\\
	\footnotesize{Universidade Federal de Alagoas}\\
	\footnotesize{Maceió, Alagoas}\\
	\footnotesize{isnaldo@pos.mat.ufal.br and marcio.melo@im.ufal.br}}

\maketitle

\begin{abstract}
	 In this work we study the initial boundary value problem associated with the  coupled Schr\"odinger equations {with quadratic nonlinearities, that appears in nonlinear optics}, on the half-line. We obtain local well-posedness for data {in Sobolev spaces} with low regularity, by using a forcing problem on the full line with a presence of a forcing term in order to apply the Fourier restriction method of Bourgain. The crucial point in this work is the new bilinear estimates on the classical Bourgain spaces $X^{s,b}$ with $b<\frac12$, jointly with bilinear estimates in adapted Bourgain spaces that will used to treat the traces of nonlinear part of the solution. Here the understanding of the dispersion relation is the key point in these estimates, where the set of regularity depends strongly of the constant $a$ measures the scaling-diffraction magnitude indices.  
\end{abstract}


\section{Introduction}

In the last years the study of initial boundary value problems for nonlinear dispersive models on the half-lines has given attention of many researchers (see \cite{Cav1} for a survey about the topic.) This turns out to be a rather challenging problem, since the more natural techniques used in standard domains, as $\R^n$ or $\mathbb T$ does not work directly in the context of the half-lines due the lack of Fourier transform.

\subsection{Presentation of the model}
Consider the quadratic interactions 
\begin{equation}\label{system}
\left\{\begin{array}{l}{i \partial_{t} u(x, t)+p \partial_{x}^{2} u(x, t)+\overline{u}(x, t) v(x, t)=0, \quad x \in \mathbb{R}, t \geq 0}; \\ {i \sigma \partial_{t} v(x, t)+q \partial_{x}^{2} v(x,t)+ u^{2}(x, t)=0}, \quad x \in \mathbb{R}, t \geq 0; \\ {u(x, 0)=u_{0}(x), \quad v(x, 0)=v_{0}(x)}, \quad x \in \mathbb{R}.\end{array}\right.
\end{equation}
Physically, according to the article \cite{menyuk-1994}, 
the complex functions $u$ and $v$ represent amplitude packets of the first and second harmonic of an optical wave, respectively. The values of $p$ and $q$ may be $1$ or $-1$, depending on the signals provided between the scattering-diffraction ratios and the positive constant $\sigma$ measures the scaling-diffraction magnitude indices. In recent years, interest in nonlinear properties of optical materials has attracted attention of physicists and mathematicians. Many researches suggest that by exploring the nonlinear reaction of the matter, the bit-rate capacity of optical fibers can be considerably increased and in consequence an improvement in the speed and economy of data transmission and manipulation. In the work \cite{CHEN2020} and \cite{numerical} current results are presented regarding numerical simulation of solutions for system (\ref{system}). 

{Another application for the system \eqref{1.a} is related to the Raman amplification in a plasma}. The study of laser-plasma interactions is an active area of interest. The main goal is to simulate nuclear fusion in a laboratory. In order to simulate numerically these experiments, we need some accurate models. The kinetic ones are the most relevant but very difficult to deal with practical computations. The fluids ones like bi-fluid Euler–Maxwell system seem more convenient but still inoperative in practice because of the high frequency motion and the small wavelength involved in the problem. This is why we need some intermediate models that are reliable from a numerical viewpoint \cite{colin2009stability}. For more details on wave propagation and processes Harmonic generation in nonlinear media see \cite{butcher1990}.




\subsection{Results  on $\R^n$ and $\mathbb T$ }

In the mathematical context N. Hayashi, T. Ozawa and K. Tanaka in \cite{hayashi2013} obtained local well-posedness for the Cauchy problem (\ref{1.a}) on the spaces $L^2(\mathbb{R}^n)\times L^2(\mathbb{R}^n)$ for $n\leq 4$ and $H^1(\mathbb{R}^n)\times H^1(\mathbb{R}^n)$ for $n\leq 6$. On the paper \cite{barbosa2018} the first author obtained local well posedness for the model posed on real line by assuming low regularity assumptions. In \cite{li2014recent} the time decay estimates of small solutions to the systems under the mass resonance condition in 2-dimensional space  was revised. In \cite{Hoshino} was obtained the global existence of analytic solutions in space dimensions $n\geq 3$, under the mass resonance condition for sufficiently small Cauchy data with exponential decay.

Regarding to qualitative properties of Cauchy problem solutions (\ref{1.a}), Linares and  Angulo  \cite{pava-2007} studied existence of periodic pulses and the stability and instability of such solitons. On the context of real line $\R$, Lopes \cite{lopes2005} obtained existence and stability os solitary waves for the system, by using variational methods.
	 In \cite{yew-2000} was obtained conditions for the existence of multipulses as well
	as a description of their geometry.  Also, on a recent paper \cite{CO} the authors obtained  formation of singularities and blow-up in the
	 $L^2(\R^n)$-(super)critical case and derived several stability results
	 concerning the ground state solutions of this system. In \cite{Hayashi 2} Hayashi,  Li and Ozawa studied the scattering theory for the system. Finally, we cite the recent works \cite{Pastor1} and \cite{Pastor2} for the dynamic of solutions in dimension $n=5$.

 \begin{remark}
 	 We empathize that it is sufficient to analyze the case $ p = q = 1 $ and observe that the other cases adapt to this one. In fact, the case $ p = q = -1 $ has the same behavior as the case $ p = q = 1 $. The case $ p = -q = 1 $ and $ p = -q = -1 $ ignores the variation of $ 0 <\sigma $ and is equivalent to the case $ p = q = 1 $ and $ \sigma> 2 $. 
 	
 \end{remark}

\subsection{Setting of the problem}This work is dedicated to the study the initial boundary value problem associated to system \eqref{system} on the half-line, more precisely
\begin{equation}\label{1.a}
	\begin{cases}
		i\partial_t u(x,t)+\partial^2_x u(x,t) + \bar{u}(x,t)v(x,t)=0, & x\in (0,+\infty),\; t\in(0,T),\\
		i\partial_t v(x,t)+a\partial^2_x v(x,t)+  u^2(x,t)=0,& x\in (0,+\infty),\; t\in (0,T),\\
				u(x,0)=u_0(x),\quad  v(x,0)=v_0(x),& x\in (0,+\infty),\\
		u(0,t)=f(t),\ v(0,t)=g(t),& t\in(0,T),
	\end{cases}
\end{equation}
where $u$ and $v$ are complex valued functions, where $a >0$. The model \eqref{1.a} is given by the nonlinear coupling of two dispersive equations of Schr\"odinger type through the quadratic terms
	\begin{equation}
		N_1(u,v)=\overline{u}\cdot v \ \mbox{ and } N_2(u,v)=u^2.
	\end{equation}

An important point in this model is the fact that the functional mass is not conserved, since some bad  terms of boundary appear in the mass functional. More precisely,  define the functional of mass for the system \eqref{1.a} by  $$\mathcal{M}(t)=\|u(t)\|^2_{L^2_x(\R^+)}+\|v(t)\|^2_{L^2_x(\R^+)}.$$

Formally, by multiplying the first equation of the system \eqref{1.a} by $\overline u$ and the second equation by $\overline v$, integrating by parts, taking the imaginary part and using Im $(\overline{u}^2v)=-\text{Im}( u^2\overline{v})$,  we get 
\begin{equation}\label{mass}
\mathcal{M}(t)=\mathcal{M}(0)+\text{Im}\int_0^t\overline u(0,s)\partial_xu(0,s)ds+a\text{Im}\int_0^t\overline v(0,s)\partial_xv(0,s)ds.
\end{equation}

This identity suggesters on the case of homogeneous boundary conditions a global result on the space $L^2(\R^+)\times L^2(\R^+)$.

\subsection{About the physics parameter $a$} In this model  the effect of the dispersion relation depends strongly of the values of $a$ (where $1/a$ denotes the mass of particle),  where this value interferes directly on the dynamic of the model.  More precisely, let the following frequencies and resonances variables
\begin{equation}\label{dispersion}
\left\lbrace 
\begin{array}{lll}
\tau=\tau_1+\tau_2 & \xi=\xi_1+\xi_2&     \\
\omega=\tau+\xi^2, & \omega_1=\tau_1-\xi_1^2,& \omega_2=\tau_2+a\xi_2^2.
\end{array}
\right.
\end{equation}
Then we have the following dispersion relation
$$
|\omega-\omega_1-\omega_2|= |\xi^2+\xi_1^2-a\xi_2^2|.$$ 
\begin{itemize}
\item On the case $0<a<\frac12$ we have the relation $ 	|\omega-\omega_1-\omega_2|\geq (1-2a)(\xi^2+\xi_1^2)	$. While on the case $a>\frac{1}{2}$, we get $|\omega-\omega_1-\omega_2|\geq 2|\xi-\mu_a \xi_2|\cdot |\xi-(1-\mu_a)\xi_2|, \mbox{ where }\ \mu_a=\frac{1-\sqrt{2a-1}}{2}.$ Note that in both cases we have a good relation dispersion, in the sense that it is possible to control the frequencies with the modulations. These situation  on the physical context is known as the non resonant mass condition ($a\neq \frac12$).

\item On the more critical case $a=\frac12$ we have not a good dispersion relation, with avoid a local result is a more larger region of regularity. This last situation, on the physical context is knows as the mass resonance condition.
\end{itemize}

In this context we call  mass resonance condition the case $a=\frac12$, while the mass nonresonance condition on the case $a\neq\frac12$.


\subsection{Functional spaces for the initial-boundary data}
Now we discuss appropriate functional spaces for the initial and boundary data, examining again the behavior of solutions of the linear problem on the line $\mathbb{R}$ as motivation.

On the line $\mathbb{R},$ we define the $L^{2}$-based inhomogeneous Sobolev spaces $H^{s}(\mathbb{R})$ equipped with the norm $\|\phi\|_{H^{s}(\mathbb{R})}=\left\|\langle\xi\rangle^{s} \widehat{\phi}(\xi)\right\|_{L^{2}(\mathbb{R})},$ where $\langle\xi\rangle=(1+|\xi|^2)^{1/2}$ and $\hat \phi$
denotes the Fourier transform of $\phi$. The operator $e^{i at \partial_{x}^{2}}$ 
 denotes the linear homogeneous solution group associated to the linear  Schr\"odinger  equations, respectively, posed on $\mathbb{R}$
i.e.,
$
e^{i at \partial_{x}^{2}} \phi(x)=\frac{1}{2 \pi} \int_{\mathbb{R}} e^{i x \xi} e^{-i at \xi^{2}} \widehat{\phi}(\xi) d \xi
.$
Some important time localized smoothing effects for the unitary groups $e^{i a t \partial_{x}^{2}}$ can be found in  \cite{KPV}. More specifically, we have the following estimate:
$$
\left\|\psi(t) e^{i at \partial_{x}^{2}} \phi(x)\right\|_{\mathcal{C}\left(\mathbb{R}_{x} ; H^{(2 s+1) / 4}\left(\mathbb{R}_{t}\right)\right)} \leq c\|\phi\|_{H^{s}(\mathbb{R}_x)},
$$
where $\psi(t)$ is a localized smooth cutoff function. This smoothing effect suggests that for data $(u_0,v_0,f,g)$ in the IBVP \eqref{1.a} is natural to consider the following hypothesis: we put
$$\mathcal{H}_{+}^{\kappa, s}:=H^{\kappa}\left(\mathbb{R}^{+}_x\right) \times H^{s}\left(\mathbb{R}^{+}_x\right) \times H^{(2 \kappa+1) / 4}\left(\mathbb{R}^{+}_t\right) \times H^{(2s+1) / 4}\left(\mathbb{R}^{+}_t\right).
$$

We fix a cutoff function $\psi$ in $C^{\infty}_0$  such that $0 \leq \psi(t) \leq 1,$
\begin{equation}
	\psi(t)=
	\begin{cases}
		1, \ \ \mbox{ if } \ \ |t|\leq 1,\\
		0,\ \ \mbox{ if }\ \ |t|\geq 2
	\end{cases}
\end{equation}
and $\psi_T(t)=\psi\left(\frac{t}{T}\right)$.

As far as we know, the local well-posedness for the system (\ref{1.a}), on the half-line, was never considered previously.
\subsection{Main Results}
Our main  local well-posedness result   is the following  statement.

\begin{theorem}\label{teo1}
	 Let the Sobolev index pair $(\kappa, s)$ verifying $s\neq \frac12$ and $\kappa \neq \frac12$ and
\begin{itemize}
	\item [(i)] $|\kappa|-1/2\leq s<\min\{\kappa+1/2,\ 2\kappa+1/2,1\}\;\text{and}\ \kappa<1\ \  \text{for}\; a>\frac12$ (first non resonant case);\\
\item [(ii)] $0\leq\kappa =s <1\; for\; a =\frac12\  (\text{resonant\ case});$\\
\item [(iii)] $\max\{-\frac12,|\kappa|-1\}\leq s<\min\{\kappa+1, \ 2\kappa+1,1 \}\; \text{and}\  \kappa<1\  \text{for}\; 0<a<\frac12$ (second non resonant case). 
For any $a>0$ and $(u_0,v_0)\in H^{\kappa}(\R^+)\times H^{s}(\R^+)$ and $(f,g)\in H^{\frac{2\kappa+1}{4}}(\R^+)\times H^{\frac{2s+1}{4}}(\R^+)$, verifying the additional compatibility conditions
\end{itemize}
\begin{equation}
\begin{cases}
	u(0)=f(0),&\text{for}\  \kappa>\frac12;
	\\
	v(0)=g(0),&\text{for}\  s>\frac12.
	\end{cases}
	\end{equation}

		 Then there exist a positive time $T=T\left(\|u_0\|_{H^{\kappa}(\R^+)}, \|v_0\|_{H^s(\R^+)}, \|f\|_{H^{\frac{2\kappa+1}{4}}(\R^+) }, \|g\|_{H^{\frac{2s+1}{4}}(\R^+) },a\right)$ and a distributional solution $(u(t), v(t))$ for the initial boundary value problem (\ref{1.a}) on the classes
	\begin{equation}
		u\in C\left([0,T]; H^{\kappa}(\mathbb{R}^+)\right) \ \ \mbox{ and } \ \ v\in C\left([0,T]; H^{s}(\mathbb{R}^+)\right).
	\end{equation}
	
	Moreover, the map $(u_0,v_0)\longmapsto (u(t), v(t))$ is locally Lipschitz from $H^{\kappa}(\mathbb{R}^+)\times H^s(\mathbb{R}^{+})$ into

	$C\left([0,T]; H^{\kappa}(\mathbb{R}^+)\times H^s(\mathbb{R}^{+})\right)$.
\end{theorem}

The approach used to prove this result is based on the arguments introduced in \cite{CK}. The main idea to solve the IBVP \eqref{1.a} is the construction of an auxiliary forced IVP in the line $\mathbb{R},$ analogous to \eqref{1.a}; more precisely:
\begin{equation}
\left\{\begin{array}{ll}
i \partial_{t} u(x, t)+\partial_{x}^{2} u(x, t)+\bar{u}(x, t) v(x, t)=\mathcal{T}_{1}(x) h_{1}(t), & (x, t) \in \mathbb{R} \times(0, T) \\
i \partial_{t} v(x, t)+a \partial_{x}^{2} v(x, t)+ u^{2}(x, t)=\mathcal{T}_{2}(x) h_{2}(t), & (x, t) \in \mathbb{R} \times(0, T) \\
u(x, 0)=\widetilde{u}_{0}(x), \quad v(x, 0)=\widetilde{v}_{0}(x), & x \in \mathbb{R}
\end{array}\right.
\end{equation}
where $\mathcal{T}_{1}, \mathcal{T}_{2}$ are appropriate distributions supported in $\mathbb{R}^{-}, \tilde{u}_{0}, \tilde{v}_{0}$ are nice extensions of $u_{0}$ and $v_{0}$ in $\mathbb{R}$ and the boundary forcing functions $h_{1}, h_{2}$ are selected to ensure that
$$
\widetilde{u}(0, t)=f(t) \quad \text { and } \quad \widetilde{v}(0, t)=g(t)
$$
for all $t \in(0, T)$.

Now, as the consequence of  Theorem \ref{teo1} and by using the  functional mass \eqref{mass} we have the following result.
\begin{theorem}\label{cor}
	For any $a>0$ and $(u_0,v_0)\in L^2(\R^+)\times L^2(\R^+)$. Then, the
	corresponding local solution on the classes
	\begin{equation}
	u\in C\left([0,T]; L^{2}(\mathbb{R}^+)\right) \ \ \mbox{ and } \ \ v\in C\left([0,T]; L^{2}(\mathbb{R}^+)\right).
	\end{equation} of the IBVP \eqref{1.a}  with homogeneous boundary conditions (i.e.
	$f = g = 0$) can be extended for all time interval $[0; T]$, for any time $T > 0$. 
\end{theorem}




Now we describe the structure of the work. {Section \ref{2} is devoted to  summarize some preliminary results}. Sections \ref{section3} and \ref{section4} we will treat the Duhamel boundary forcing operator classes associated to linear Schr\"odinger equation.  Section \ref{bilinear} we will get the bilinear estimates for the coupling terms nonlinear. In Section \ref{6} we will show the proof of local result for local result and Section \ref{7} the global result. Finally, in Appendix we will prove a more technical lemma concern the bilinear estimate in adapted Bourgain spaces. 
\section{Preliminary results}\label{2}

A important point in the approach used here is the understanding of the how capture the dispersive effect smoothing caused by the nonlinear part of a dispersive equation. To illustrate this situation we consider the Cauchy problem  of the form
\begin{equation}\label{2.1}
	i\partial_tu(x,t)+ia\partial_x^2u(x,t)=F(u(x,t)),
\end{equation}
where $F$ is a nonlinear function.

The Cauchy Problem for (\ref{2.1})  is rewritten as the following integral equation
\begin{equation}\label{2.2}
	u(t)=U_{a}(t)u_0-i\displaystyle\int_0^t U_{a}(t-t')F(u(t'))dt',
\end{equation}
where $U_{a}(t)=e^{-ita\partial_x^2}$ is the group that solves the linear part of (\ref{2.1}).

There are a number of ways to capture this dispersive smoothing effect, but one particularly convenient way is via the $X^{s,b}$ spaces (also known as Fourier restriction spaces, Bourgain spaces, or dispersive Sobolev spaces), introduced by Bourgain in \cite{bourgain-1993}. The full name of these
 is $X_{\tau=a\xi^2}^{s, b}\left(\mathbb{R} \times \mathbb{R}\right),$ thus these spaces take $\mathbb{R} \times \mathbb{R}$ as their domain and are adapted to a single characteristic hypersurface $\tau=a\xi^2 .$ Roughly speaking, these spaces are to dispersive equations as Sobolev spaces are to elliptic equations (fore more details see \cite{linares} and \cite{tao}).

Let $X_a^{s,b}$ be the completion of $\mathcal{S}(\mathbb{R}^2)$ with respect to the norm
\begin{equation}\label{2.3}
	\begin{array}{ll}
		\left\| f \right\|_{X_a^{s,b}}&=\left\| \langle\xi\rangle^s\langle\tau+a\xi^2\rangle^b \widehat{f}(\tau,\xi) \right\|_{L^2_{\tau}L^2_{\xi}}.
	\end{array}
\end{equation}

The following lemma was proved while establishing the local well-posedness of the Zakharov system by Ginibre, Tsutsumi and Velo in \cite{ginibre-1997}.

\begin{lemma}\label{l2.1}
	Let $-\frac{1}{2}<b'\leq 0\leq b\leq b'+1$, $\psi$ a cutoff function and $T\in [0,1]$. Then for $F\in X_a^{s, b'}$ we have
	\begin{equation}\label{2.4}
		\left\|\psi_1(t)U_{a}(t)u_0 \right\|_{X_a^{s,b}}\leq c \left\|u_0 \right\|_{H^s},
	\end{equation}
	\begin{equation}\label{2.5}
		\left\|\psi_T(t)\displaystyle\int_0^t U_{a}(t-t')F(u(t'))dt'\right\|_{X_a^{s,b}}\leq c T^{1+b'-b} \left\|F \right\|_{X_a^{s,b'}}.
	\end{equation}
\end{lemma}

\begin{proof}
	See Lemma 2.1 in \cite{ginibre-1997}.
\end{proof}

A delicate point in IBVPS posed on the half-line is to treat the traces of solutions. To treat this we need to work with the following adapted Bourgain spaces $W_a^{s,b}$ given by
\begin{equation}
\|u\|_{W_a^{s, b}}=\left(\iint\langle\tau\rangle^{\frac{s}{2}}\left\langle\tau-a\xi^{2}\right\rangle^{2 b}|\hat{u}(\xi, \tau)|^{2} d \xi d \tau\right)^{\frac{1}{2}}
\end{equation}

The traces estimate for the Duhamel operator obtained in \cite{cavalcante} read as follows.
 \begin{lemma}
 Let $\mathcal{S}_a u=\displaystyle\int_0^t U_a(t-t')u(x,t)dt'$.	The following estimate is valid for $-\frac12<d_1<0$.
 	\begin{equation}
 	\|\psi(t) \mathcal{S}_a u(x, t)\|_{\mathcal{C}\left(\mathbb{R}_{x} ; H^{(2 s+1) / 4}\left(\mathbb{R}_{t}\right)\right)} \leq\left\{\begin{array}{ll}
 		c\|u\|_{X_a^{s, d_{1}}}, & \text { if }-\frac{1}{2}<s \leq \frac{1}{2} \\
 		c\left(\|u\|_{U_a^{s, d_{1}}}+\|u\|_{X_a^{s, d_{1}}}\right), & \text { for all } s \in \mathbb{R}.
 	\end{array}\right.
 	\end{equation}
 \end{lemma}

We finish this section with the following elementary integral estimates which will be used to estimate the nonlinear terms in Section \ref{bilinear}.

\subsection{Elementary integral estimate}
Now we enunciate some elementary integral estimate, where the proofs can be view in \cite{Holmer}.
\begin{lemma}\label{lemagtv}
	Let  $b_1,b_2$, such that $b_1+b_2>\frac{1}{2}$ and $b_1,b_2< \frac{1}{2}$. Then
	\begin{equation*}
	\int\frac{dy}{\langle y-\alpha\rangle^{2b_1}\langle y-\beta\rangle^{2b_2}}\leq \frac{c}{\langle \alpha-\beta\rangle^{2b_1+2b_2-1}}.
	\end{equation*}
\end{lemma}
\begin{lemma}\label{lemanovo}
	If $b>\frac{1}{2}$, then
	\begin{equation*}
	\int_{-\infty}^{\infty}\frac{dx}{\langle\alpha_0+\alpha_1x+x^2\rangle^b}\leq c.
	\end{equation*}
\end{lemma}
\begin{lemma}\label{Hol}
	If $b<\frac{1}{2}$, then
	\begin{equation*}
	\int_{|x|<\beta}\frac{dx}{\langle x\rangle^{4b-1}|\alpha-x|^{\frac{1}{2}}}\leq c\frac{(1+\beta)^{2-4b}}{\langle\alpha\rangle^{\frac{1}{2}}}.
	\end{equation*}
\end{lemma}

\subsection{Riemann-Liouville fractional integral operator}

For Re $\alpha>0,$ the tempered distribution $\frac{t_{+}^{\alpha-1}}{\Gamma(\alpha)}$ is defined as a locally integrable function by
$$
\left\langle\frac{t_{+}^{\alpha-1}}{\Gamma(\alpha)}, f\right\rangle:=\frac{1}{\Gamma(\alpha)} \int_{0}^{+\infty} t^{\alpha-1} f(t) {d} t.
$$
For Re $\alpha>0,$ integration by parts implies that
$$
\frac{t_{+}^{\alpha-1}}{\Gamma(\alpha)}=\partial_{t}^{k}\left(\frac{t_{+}^{\alpha+k-1}}{\Gamma(\alpha+k)}\right)
$$
for all $k \in \mathbb{N} .$ This expression allows to extend the definition, in the sense of distributions, of $\frac{t_{+}^{\alpha-1}}{\Gamma(\alpha)}$ to all $\alpha \in \mathbb{C}$
If $f \in C_{0}^{\infty}\left(\mathbb{R}^{+}\right),$ we define
$$
\mathcal{I}_{\alpha} f=\frac{t_{+}^{\alpha-1}}{\Gamma(\alpha)} * f
$$
Thus, for Re $\alpha>0$
$$
\mathcal{I}_{\alpha} f(t)=\frac{1}{\Gamma(\alpha)} \int_{0}^{t}(t-s)^{\alpha-1} f(s) \mathrm{d} s
$$
and notice that
$$
\mathcal{I}_{0} f=f, \quad \mathcal{I}_{1} f(t)=\int_{0}^{t} f(s) \mathrm{d} s, \quad \mathcal{I}_{-1} f=f^{\prime} \quad \text { and } \quad \mathcal{I}_{n} \mathcal{I}_{\beta}=\mathcal{I}_{\alpha+\beta}.
$$

\section{The Duhamel Boundary Forcing Operator}\label{section3}
We now introduce the   Duhamel boundary forcing operator similar to the introduced in \cite{Holmer}. For $f\in C_0^{\infty}(\mathbb{R}^+)$, define the boundary forcing operator
\begin{eqnarray*}
	\mathcal{L}_af(x,t)&=&-iC\int_0^te^{i(t-t')a\partial_x^2}\delta_0(x)\mathcal{I}_{-\frac{1}{2}}f(t')dt'\\
	&=&-i\frac{Ce^{-i\pi/4}}{2\sqrt{a\pi}}\int_0^t(t-t')^{-\frac{1}{2}}e^{\frac{ix^2}{4a(t-t')}}\mathcal{I}_{-\frac{1}{2}}f(t')dt'\\
		&=&\frac{C}{2 \sqrt{a\pi}} e^{i\frac{3\pi}{4}}\int_0^t(t-t')^{-\frac{1}{2}}e^{\frac{ix^2}{4a(t-t')}}\mathcal{I}_{-\frac{1}{2}}f(t')dt',
\end{eqnarray*}
where we have used the formula
\begin{equation*}
\mathcal{F}_x\left( \frac{e^{-i \frac{\pi}{4}\text{sgn}\ y} }{2|y|^{1/2}\sqrt{\pi}}e^{\frac{ix^2}{4y}}\right)(\xi)=e^{-iy\xi^2},\ \forall\ y\in \mathbb{R}.
\end{equation*}

For $\mbox{Re }\alpha>0$, we set
\begin{equation}\label{eq:IO}
\mathcal{I}_{\alpha}f(t)=\frac{1}{\Gamma(\alpha)}\int_0^t(t-s)^{\alpha-1}f(s) \; ds.
\end{equation}

From this definition, we see that
\begin{equation}\label{forcante00}
\begin{cases}
(i\partial_t+a\partial_x^2)\mathcal{L}_af(x,t)=C\delta_0(x)\mathcal{I}_{-\frac{1}{2}}f(t),& x,t\in\mathbb{R},\\
\mathcal{L}_af(x,0)=0,& x\in \R.
\end{cases}
\end{equation}

By choosing $C=2e^{-\frac34\pi i}\sqrt a$, we have that
$$\mathcal{L}_af(0,t)=f(t).$$

Thus, for $f\in C_0^{\infty}(\R^+)$, set $u(x,t)=e^{-iat\partial_x^2}\phi(x)+\mathcal{L}(f-e^{-iat\partial_x^2}\phi(x)\big|_{x=0})$. Then,  $u(x,t)$ is continuous in $x$. Thus $u(0,t)=f(t)$ and $u(x,t)$ solves the problem
\begin{equation}\label{forcante}
\begin{cases}
(i\partial_t+a\partial_x^2)u(x,t)=C\delta(x)\mathcal{I}_{-1/2}(f-e^{-iat\partial_x^2}\phi(x)\big|_{x=0}),& (x,t)\in\mathbb{R},\\
u(x,0)=\phi(x),& x\in\mathbb{R},\\
u(0,t)=f(t),& t\in \R.
\end{cases}
\end{equation}
This would suffice to solve the linear analogue of the half-line problem.

\section{The Duhamel boundary forcing operator classes associated to linear Schr\"odinger equation}\label{section4}

In order to get our results in a larger class of index regularity, we need to work with a class of boundary forcing operators in order to obtain the required estimates for the second-order derivative of traces. In this way, we define the generalization of operators $\mathcal{L}_a$.

For $\lambda \in \mathbb{C}$ such that Re $\lambda>-2$ and $f \in C_{0}^{\infty}\left(\mathbb{R}^{+}\right)$ define
$$
\mathcal{L}_a^{\lambda} f(x, t)=\left[\frac{x_{-}^{\lambda-1}}{\Gamma(\lambda)} * \mathcal{L}_a\left(\mathcal{I}_{-\frac{\lambda}{2}} f\right)(\cdot, t)\right](x)
$$
with $\frac{x_{-}^{\lambda-1}}{\Gamma(\lambda)}=\frac{(-x)_{+}^{\lambda-1}}{\Gamma(\lambda)}.$ 
These definition implies
\begin{equation}\label{forc}
\left(i \partial_{t}+a\partial_{x}^{2}\right) \mathcal{L}_{a}^{\lambda} f(x, t)=\frac{C}{\Gamma(\lambda)} x_{-}^{\lambda-1} \mathcal{I}_{-\frac{1}{2}-\frac{\lambda}{2}} f(t).
\end{equation}

If Re $\lambda>0,$ then
\begin{equation}\label{def0}
\mathcal{L}_a^{\lambda} f(x, t)=\frac{1}{\Gamma(\lambda)} \int_{x}^{+\infty}(y-x)^{\lambda-1} \mathcal{L}\left(\mathcal{I}_{-\frac{\lambda}{2}} f\right)(y, t) d y.
\end{equation}
For Re $\lambda>-2,$ using \eqref{forc} we obtain
\begin{equation}\label{alt}
\begin{aligned}
\mathcal{L}_{a}^{\lambda} f(x, t) &=\frac{1}{\Gamma(\lambda+2)} \int_{x}^{+\infty}(y-x)^{\lambda+1} \partial_{y}^{2} \mathcal{L}\left(\mathcal{I}_{-\frac{\Lambda}{2}} f\right)(y, t) d y \\
&=-\frac1a\int_{x}^{+\infty} \frac{(y-x)^{\lambda+1}}{\Gamma(\lambda+2)}\left(i \partial_{t} \mathcal{L} \mathcal{I}_{-\frac{\lambda}{2}} f\right)(y, t) d y+\frac{C}{a} \frac{x_{-}^{\lambda+1}}{\Gamma(\lambda+2)} \mathcal{I}_{-1 / 2-\lambda / 2} f(t)
\end{aligned}
\end{equation}
Notice that $\left.\frac{x_{\pm}^{\lambda-1}}{\Gamma(\lambda)}\right|_{\lambda=0}=\delta_{0},$ then $\mathcal{L}_{a}^{0} f(x, t)=\mathcal{L}_a f(x, t)$.

From Lemma \ref{alt} it follows that $\mathcal{L}_{a}^{\lambda} f(x, t)$ is well defined for $\lambda>-2$ for $t \in[0,1]$ Moreover, the dominated convergence theorem and Lemma 3.2 imply that, for fixed $t \in[0,1]$ and $\operatorname{Re} \lambda>-1,$ the function $\mathcal{L}^{\lambda} f(x, t)$ is continuous in $x$ for all $x \in \mathbb{R}$. The following result establishes the values of $\mathcal{L}_{\pm}^{\lambda} f(x, t)$ at $x=0$
\begin{lemma}\label{trace} If Re $\lambda>-1$ and $f \in C_{0}^{\infty}\left(\mathbb{R}^{+}\right),$ then $\mathcal{L}_{a}^{\lambda} f(0, t)=\sqrt{a}e^{i \frac{ \lambda \pi}{4}}f(t)$.\end{lemma}
\begin{proof}
	By using \eqref{alt}, we have that
	$$\mathcal{L}_a^{\lambda} f(0, t)=-\frac{1}{a}\int_{0}^{+\infty} \frac{y^{\lambda+1}}{\Gamma(\lambda+2)} i \partial_{t} \mathcal{L}_{a}\left(\mathcal{I}_{-\frac{\lambda}{2}} f\right)(y, t) d y.$$
	By complex differentiation under the integral sign,  we have that $\mathcal{L}_a^{\lambda} f(0, t)$ is analytic
	in $\lambda,$ for $\operatorname{Re} \lambda>-1 .$ By analyticity, we shall only compute $\mathcal{L}_a^{\lambda} f(0, t)$ for $0<\lambda<2$.
	For the computation in the range $0<\lambda<2,$ we use the representation \eqref{def0} to obtain 
	$$
	\begin{aligned}
	\mathcal{L}_a^{\lambda} f(0, t)=& \frac{1}{\Gamma(\lambda)} \int_{0}^{+\infty} y^{\lambda-1} \mathcal{L}_a\left(\mathcal{I}_{-\frac{\lambda}{2}} f\right)(y, t) d y \\
=	& \frac{1}{\Gamma(\lambda)} \int_{0}^{+\infty} y^{\lambda-1} \frac{1}{\sqrt{\pi}} \int_{0}^{t}\left(t-t^{\prime}\right)^{-\frac{1}{2}} e^{\frac{i y^{2}}{4a\left(t-t^{\prime}\right)}} \mathcal{I}_{-\frac{\lambda}{2}-\frac{1}{2}} f\left(t^{\prime}\right) d t^{\prime} d y \\
	=& \frac{1}{\Gamma(\lambda)} \int_{0}^{t}\left(t-t^{\prime}\right)^{-\frac{1}{2}} \mathcal{I}_{-\frac{1}{2}-\frac{\lambda}{2}} f\left(t^{\prime}\right) \int_{0}^{+\infty} y^{\lambda-1} \frac{1}{\sqrt{\pi}} e^{\frac{i y^{2}}{4a (t-t^{\prime})}} d y d t^{\prime}
	\end{aligned}
	$$
	Set $I=\int_{0}^{+\infty} y^{\lambda-1} e^{\frac{i y^{2}}{4a(t-t)}} d y .$ Changing variables $r=\frac{y^{2}}{4a\left(t-t^{\prime}\right)},$ then $y=r^{\frac{1}{2}} 2\sqrt{a}\left(t-t^{\prime}\right)^{\frac{1}{2}}$ and
	$d y=r^{-\frac{1}{2}}\sqrt{a}\left(t-t^{\prime}\right)^{\frac{1}{2}} d r,$ we get
	$$
	I=\int_{0}^{+\infty}\left[r^{\frac{1}{2}} 2\sqrt{a}\left(t-t^{\prime}\right)^{\frac{1}{2}}\right]^{\lambda-1} r^{-\frac{1}{2}} e^{i r}\left(t-t^{\prime}\right)^{\frac{1}{2}} d r=2^{\lambda-1}\sqrt{a}\left(t-t^{\prime}\right)^{\frac{1}{2}} \int_{0}^{+\infty} r^{\frac{\lambda}{2}-1} e^{i r} d r
	$$
	By a change of contour,
	$$
	I=2^{\lambda-1}\sqrt{a}\left(t-t^{\prime}\right)^{\frac{1}{2}} i^{\frac{\lambda}{2}} \int_{0}^{+\infty} r^{\frac{\lambda}{2}-1} e^{-r} d r=2^{\lambda-1}\sqrt{a}\left(t-t^{\prime}\right)^{\frac{1}{2}} i^{\frac{\lambda}{2}} \Gamma\left(\frac{\lambda}{2}\right), \text { for } \lambda \in(0,2)
	$$
	Using the formula
	$$
	\frac{\Gamma\left(\frac{\lambda}{2}\right)}{\Gamma(\lambda)}=\frac{2^{1-\lambda} \sqrt{\pi}}{\Gamma\left(\frac{\lambda}{2}+\frac{1}{2}\right)}
	$$
	for $\lambda \in \mathbb{R}^{+},$ we obtain
	$$
	\begin{aligned}
	\mathcal{L}_a^{\lambda} f(0, t) &=\frac{2^{\lambda-1}\sqrt{a}}{\sqrt{\pi}} \frac{\Gamma\left(\frac{\lambda}{2}\right)}{\Gamma(\lambda)} i^{\frac{ \lambda}{2}} \int_{0}^{t}\left(t-t^{\prime}\right)^{\frac{\lambda}{2}-\frac{1}{2}} \mathcal{I}_{-\frac{\lambda}{2}-\frac{1}{2}} f\left(t^{\prime}\right) d t^{\prime} \\
	&=\frac{\sqrt{a}}{\Gamma\left(\frac{\lambda}{2}+\frac{1}{2}\right)} i^{\frac{ \lambda}{2}} \int_{0}^{t}\left(t-t^{\prime}\right)^{\frac{\lambda}{2}-\frac{1}{2}} \mathcal{I}_{-\frac{\lambda}{2}-\frac{1}{2}} f\left(t^{\prime}\right) d t^{\prime}=i^{\lambda/2}\sqrt{a} \mathcal{I}_{\frac{\lambda}{2}+\frac{1}{2}} \mathcal{I}_{-\frac{\lambda}{2}-\frac{1}{2}} f(t)=\sqrt{a}e^{i \frac{3 \lambda \pi}{4}} f(t).
	\end{aligned}
	$$
	
	\end{proof}

We finish this section with some estimates for the Duhamel boundary forcing operator class $\mathcal{L}_a$, whose proof is similar to proof of Lemma 6.2 in \cite{cavalcante}.
\begin{lemma}\label{estimate} Let $s \in \mathbb{R}$ and $f \in C_{0}^{\infty}\left(\mathbb{R}^{+}\right) .$ The following estimates are valid:
\begin{itemize}
	\item[(a)] (Space traces) $\left\|\mathcal{L}_{a}^{\lambda} f(x, t)\right\|_{\mathcal{C}\left(\mathbb{R}_{t} ; H^{s}\left(\mathbb{R}_{x}^{+}\right)\right)} \leq c\|f\|_{H_{0}^{(2 s+1) / 4}\left(\mathbb{R}^{+}\right)}$ whenever $s-\frac{3}{2}<$
$\lambda<\min \left\{s+\frac{1}{2}, \frac{1}{2}\right\}$ and $\operatorname{supp}(f) \subset[0,1]$.
\item[(b)] (Time traces) $\left\|\psi(t) \mathcal{L}_{a}^{\lambda} f(x, t)\right\|_{\mathcal{C}\left(\mathbb{R}_{x} ; H_{0}^{(2 s+1) / 4}\left(\mathbb{R}_{t}^{+}\right)\right)} \leq c\|f\|_{H_{0}^{(2 s+1) / 4}\left(\mathbb{R}^{+}\right)}$ whenever
$-1<\lambda<1$
\item[(c)](Bourgain spaces) $\left\|\psi(t) \mathcal{L}_{a}^{\lambda} f(x, t)\right\|_{X^{s, b}} \leq c\|f\|_{H_{0}^{(2 s+1) / 4}\left(\mathbb{R}^{+}\right)}$ whenever $s-\frac{1}{2}<$
$\lambda<\min \left\{s+\frac{1}{2}, \frac{1}{2}\right\}$ and $b<\frac{1}{2}$.
\end{itemize}
\end{lemma}

\begin{remark}
	In order to all sentences of Lemma \eqref{estimate} does work it is necessary that the following restrictions of indexes 
	$$\left\{-1,s-\frac12\right\}<\lambda<\min\left\{s+\frac12,\frac12\right\}.$$ Then in the case of the operator $\mathcal{L}_a$ associated to the index $\lambda=0$, these estimates are valid on the set regularity index $-\frac12<s<\frac12$, then the use of the more general classes of boundary operators $\mathcal{L}_a^{\lambda}$ is fundamental to get results in a more larger region.
	\end{remark}

\section{Bilinear estimates for the coupling terms}\label{bilinear}

 The main results in this section are the following lemmas with the bilinear estimates for different values of $a>0$.  Each case lead us to different restrictions on the Sobolev index $s$ and $\kappa$.   

\begin{lemma}\label{l1}
	 Consider $b,d\in(3/8,\ 1/2)$. Then we have the following inequality 
	\begin{equation}\label{desigualdade1}
	\left\|\overline{u}\cdot v \right\|_{X^{\kappa,-d}}\leq c \left\|u \right\|_{X_a^{\kappa,b}}\cdot \left\|v \right\|_{X^{s,b}}, 
	\end{equation}
 in the following cases:
	\begin{itemize}
		\item $0<a<\frac12$ $(\sigma >2)$ and $|\kappa|-s<1$;
		\item $ a>\frac12$ $(0<\sigma<2)$ and $|\kappa|-s\leq\frac12$;
		\item $a=\frac12$ $(\sigma=2)$ and  $|\kappa|\leq s$.
	\end{itemize}
	
\end{lemma}

\begin{lemma}\label{l2}
		Consider $b,d\in(3/8,\ 1/2)$. Then we have the following inequality
		\begin{equation}\label{desigualdade2}
			\left\|u\cdot \tilde{u} \right\|_{X^{s,-d}_a}\leq c \left\|u \right\|_{X^{\kappa,b}}\cdot \left\|\tilde{u} \right\|_{X^{\kappa,b}},
		\end{equation}
		 in the following cases:
		 \begin{itemize}
		 	\item $0<a<\frac12$ $(\sigma >2)$ and $s< \min\left\lbrace \kappa +1,\  \ 2\kappa +1\right\rbrace$;
		 	\item $ a>\frac12$ $(0<\sigma<2)$ and $s\leq \min\left\lbrace \kappa +1/2,\  \ 2\kappa +1/2\right\rbrace $;
		 	\item $a=\frac12$ $(\sigma=2)$ and $0 \leq s\leq \kappa $. 
		 \end{itemize}
		 
\end{lemma}

The following Lemma, which the first one will be proved on the appendix, is need to complete the problem in a  more larger set of regularity.

\begin{lemma}\label{lw1}
	There exist $b=b(s,\kappa)<\frac12$ and $d=d(s,\kappa)<\frac12$ such that holds the following inequality
	\begin{equation}\label{desigualdade1w}
	\left\|\overline{u}\cdot v \right\|_{W^{\kappa,-d}}\leq c \left\|u \right\|_{X^{\kappa,b}}\cdot \left\|v \right\|_{X_a^{s,b}}, 
	\end{equation}
	in the following cases:
	\begin{itemize}
		\item $0<a<\frac12$ $(\sigma >2)$:\begin{equation*}
		 \begin{split}
		 &\{(s,\kappa)\in \R^2;\frac12<\kappa\leq2d,\  s\geq -\frac12,\  \text{and}\  |\kappa|-s<1\}\\ &\{(s,\kappa)\in \R^2;\kappa\leq-\frac12,  \text{and}\  s-|\kappa|\leq 4b\};
		 \end{split}
		 \end{equation*}
		\item $ a>\frac12$ $(0<\sigma<2)$, $\frac12<\kappa<2d$ and $|\kappa|-s\leq\frac12$;
		\item $a=\frac12$ $(\sigma=2)$, $\kappa>\frac12$ and  $|\kappa|\leq s$.
	\end{itemize}
	
\end{lemma}

The following lemma treats the second nonlinearity, its proof follow the same ideas of the previous lemma and will be omitted here.

\begin{lemma}\label{lw2}
	There exist $b=b(s,\kappa)<\frac12$ and $d=d(s,\kappa)<\frac12$ such that holds the following inequality
	\begin{equation}\label{desigualdade2w}
	\left\|u\tilde u  \right\|_{W_a^{\kappa,-d}}\leq c \left\|u \right\|_{X^{\kappa,b}}\cdot \left\|\tilde u \right\|_{X^{s,b}}, 
	\end{equation}
	in the following cases:
	 \begin{itemize}
		\item $0<a<\frac12$ $(\sigma >2)$:
		\begin{equation}
		 \begin{split}
		 &-\frac12<s< \min\left\lbrace \kappa +1,\  \ 2\kappa +1\right\rbrace;
		 \end{split}
		 \end{equation}
		\item $ a>\frac12$ $(0<\sigma<2)$ and $s\leq \min\left\lbrace \kappa +1/2,\  \ 2\kappa +1/2\right\rbrace $;
		\item $a=\frac12$ $(\sigma=2)$ and $0 \leq s\leq \kappa $. 
	\end{itemize}
	
\end{lemma}

Initially we prove the first lemma.

\begin{proof}[Proof of the Lemma \ref{l1}]

The inequality \eqref{desigualdade1} produces the following frequencies and resonances variables
\begin{equation}\label{relacao1}
	\left\lbrace 
	\begin{array}{lll}
	\tau=\tau_1+\tau_2 & \xi=\xi_1+\xi_2&     \\
	\omega=\tau+\xi^2, & \omega_1=\tau_1-\xi_1^2,& \omega_2=\tau_2+a\xi_2^2.
	\end{array}
	\right.
\end{equation}
 We split the analysis in to cases.
\begin{itemize}
	\item Case $0<a<\frac12$:

	Following the ideas of \cite{barbosa2018}
		it is enough to show that the following integral functions 
		\begin{enumerate}
			\item[] \begin{equation}\label{j1} J_1(\xi,\tau)= \dfrac{1}{\langle\tau+\xi^2\rangle^{2d}}\displaystyle\int_{\mathbb{R}}\dfrac{\langle\xi_2\rangle^{-2s+2|{\kappa}|}\chi_{\mathcal{R}_1}}{\langle\tau-(a-1)\xi_2^2-2\xi\xi_2+\xi^2\rangle^{{4b-1}}}d\xi_2;	\end{equation}
			\item[] \begin{equation}\label{j2} J_2(\xi_2,\tau_2)= \dfrac{1}{\langle\tau_2+a\xi_2^2\rangle^{2b}}\displaystyle\int_{\mathbb{R}}\dfrac{\langle\xi_2\rangle^{-2s+2|{\kappa}|}\chi_{\mathcal{R}_2}}{\langle\tau_2+2\xi^2+\xi_2^2-2\xi\xi_2\rangle^{{2b+2d-1}}}d\xi;	\end{equation}
			\item[] 	\begin{equation}\label{j3} J_3(\xi_1,\tau_1)= \dfrac{1}{\langle\tau_1-\xi_1^2\rangle^{2b}}\displaystyle\int_{\mathbb{R}}\dfrac{\langle\xi_2\rangle^{-2s+2|{\kappa}|}\chi_{\mathcal{R}_3}}{\langle\tau_1-a\xi_2^2+\xi^2\rangle^{{2b+2d-1}}}d\xi_2	\end{equation}
		\end{enumerate}  
	are bounded.
	
		In this case, we use the \eqref{relacao1} to get 
	\begin{eqnarray*}
		|\omega-\omega_1-\omega_2|&=& |\xi^2+\xi_1^2-a\xi_2^2|\\
		&\geq & |1-a|(\xi^2+\xi_1^2)-2a|\xi\xi_1|, \ \ \mbox{ (since } 0<a<\frac{1}{2})\\
		&\geq & (1-a)(\xi^2+\xi_1^2)-a(\xi^2+\xi_1^2)=(1-2a)(\xi^2+\xi_1^2).
	\end{eqnarray*}

		It follows that, $$3\max \{|\omega|,|\omega_1|, |\omega_2|\}\geq (1-2a)\max\{\xi^2,\xi_1^2\}\geq \frac{1-2a}{4} \xi_2^2.$$
		Now, suppose that $|\xi_2|\geq 1$, {then} we have
		$$\dfrac{1}{\max \{|\omega|,|\omega_1|, |\omega_2|\}}\leq  \dfrac{c}{|\xi_2|^2}.$$
	We define the following regions 
		\begin{equation}
			\mathcal{R}_1= \bigg\lbrace |\xi_2|\geq 1, |\omega|=\max \{|\omega|,|\omega_1|, |\omega_2|\}\bigg\rbrace\cup\bigg\{|\xi_2| \leq 1\bigg\}\subset \mathbb{R}^4_{\xi,\tau,\xi_2,\tau_2};
		\end{equation}
		
		\begin{equation}
			\mathcal{R}_2= \bigg\{ |\xi_2|\geq 1, |\omega_1|=\max \{|\omega|,|\omega_1|, |\omega_2|\}\bigg\}\subset \mathbb{R}^4_{\xi,\tau,\xi_2,\tau_2}
		\end{equation}
		and
		\begin{equation}
			\mathcal{R}_3= \bigg\{ |\xi_2|\geq 1, |\tau_2+a\xi^2|=\max \{|\omega|,|\omega_1|, |\omega_2|\}\bigg\}\subset \mathbb{R}^4_{\xi,\tau,\xi_2,\tau_2}.
		\end{equation}
		Let us prove that $J_1$ is bounded.  If $|\xi_2|\leq 1$, then to control $J_1$ is equivalent to get $$\dfrac{1}{\langle\omega\rangle^{2d}}\displaystyle\int_{|\xi_2|\leq 1}\dfrac{1}{\langle\tau-(a-1)\xi_2^2-2\xi\xi_2+\xi^2\rangle^{4b-1}}d\xi_2\leq c.$$ If $|\xi_2|\geq 1$, then $J_1$ is bounded by  $$\displaystyle\int_{|\xi_2|\geq 1}\dfrac{\langle\xi_2\rangle^{-2s+2|{\kappa}|-4d}\chi_{\mathcal{R}_1}}{\langle\tau-(a-1)\xi_2^2-2\xi\xi_2+\xi^2\rangle^{{4b-1}}}d\xi_2.$$ Note that $J_1$ is bounded when $|{\kappa}|-s\leq 2d<1$ since $b>3/8$.
		
		To prove that $J_2$ is bounded, it is suffices to note that the integral below is higher than $J_2$ and that converges since $|{\kappa}|-s\leq 2b$ and  $2b+2d-1>1/2 $, that is, $b<1/2$.
		$$ \displaystyle\int_{\mathbb{R}}\dfrac{\langle\xi_2\rangle^{-2s+2|{\kappa}|-4b}\chi_{\mathcal{R}_2}}{\langle\tau_2+2\xi^2+\xi_2^2-2\xi\xi_2\rangle^{{2b+2d-1}}}d\xi.$$
		
		Analogously, in a similar way, we can prove that $J_3$ is bounded, by using that $|{\kappa}|-s\leq 2b$ and $b<1/2$.
		
		\item Case $a>\frac12$: 
		We start by considering the dispersion relation 
		\begin{eqnarray*}
			|\omega-\omega_1-\omega_2|&=& |\xi^2+\xi_1^2-a\xi_2^2|\\
			&\geq &  2|\xi-\mu_a \xi_2|\cdot |\xi-(1-\mu_a)\xi_2|, \mbox{ where }\ \mu_a=\frac{1-\sqrt{2a-1}}{2}.
		\end{eqnarray*}

	Now, define the following sets
		\begin{eqnarray*}
			\mathcal{A}_1&=& \{ |\xi_2| \leq 1\}\subset \mathbb{R}^4,  \\
			\mathcal{A}_2&=& \left\{ |\xi_2|\geq 1, \left|(1-a)\xi_2-\xi\right|>\frac{2a-1}{4}|\xi_2|\right\}\subset \mathbb{R}^4\ \text{and}  \\
			\mathcal{A}_3&=& \left\{ |\xi_2|\geq 1, \left|\xi-\frac{1}{2}\xi_2\right|>\frac{2a-1}{4}|\xi_2|\right\}\subset \mathbb{R}^4. 
		\end{eqnarray*}

	We split ${\mathcal A}_3$ in three subsets given by
		\begin{eqnarray*}
			\mathcal{A}_{3,1}&=& \mathcal{A}_3 \cap \left\{ |\omega|\geq \max \{|\omega_1|, |\omega_2|\} \right\}, \\
			\mathcal{A}_{3,2}&=& \mathcal{A}_3 \cap \left\{ |\omega_2|\geq \max \{|\omega_1|, |\omega|\}\right\}\ \text{and}\\
			\mathcal{A}_{3,3}&=& \mathcal{A}_3\cap \left\{ |\omega_1|\geq \max \{|\omega|, |\omega_2|\}\right\}.
		\end{eqnarray*}
			
		Now, we define the regions $\mathcal{R}_i$.
		 
		Let $\mathcal{R}_1=\mathcal{A}_1\cup\mathcal{A}_2\cup \mathcal{A}_{3,1}$, $\mathcal{R}_2=\mathcal{A}_{3,2}$ and $\mathcal{R}_3=\mathcal{A}_{3,3}.$
		
		We will show that $J_1$ is bounded. If $|\xi_2|\leq 1$ then $J_1$ is equivalent to $$\dfrac{1}{\langle\omega\rangle^{2d}}\displaystyle\int_{|\xi_2|\leq 1}\dfrac{1}{\langle\tau-(a-1)\xi_2^2-2\xi\xi_2+\xi^2\rangle^{4b-1}}d\xi_2\leq c.$$ 
		
		If $|\xi_2|\geq 1$, in $\mathcal{A}_2$ we have  
		$$J_1 \leq \dfrac{1}{\langle\omega\rangle^{2d}}\displaystyle\int_{|\xi_2|\geq 1}\dfrac{\langle\xi_2\rangle^{-2s+2|{\kappa}|}\chi_{\mathcal{A}_2}}{\langle\tau-(a-1)\xi_2^2-2\xi\xi_2+\xi^2\rangle^{4b-1}}d\xi_2.$$ 
		Changing the variable $\eta=\tau-(a-1)\xi_2^2-2\xi\xi_2+\xi^2$,   we get $$d\eta=-2\left((1-a)\xi_2-\xi\right)d\xi_2,$$ and, using the fact $|{\kappa}|-s< 1/2$, we obtain the following 
		\begin{eqnarray*}
			\dfrac{1}{\langle\omega\rangle^{2d}}\displaystyle\int_{|\xi_2|\geq 1}\dfrac{\langle\xi_2\rangle^{-2s+2|{\kappa}|}\chi_{\mathcal{A}_2}}{\langle\tau-(a-1)\xi_2^2-2\xi\xi_2+\xi^2\rangle^{2b}}d\xi_2 &\leq & c\dfrac{1}{\langle\omega\rangle^{2d}}\displaystyle\int_{\langle\eta \rangle\leq \langle\omega \rangle}\dfrac{\langle\xi_2\rangle^{-2s+2|{\kappa}|-1}}{\langle\eta \rangle^{4b-1}}d\eta\\
			&\leq & c\dfrac{\langle\omega \rangle^{2-4b}}{\langle\omega \rangle^{2d}}\leq c \langle\omega \rangle^{2-4b-2d}\leq c,
		\end{eqnarray*}
	since $b,d>3/8$.

		Now, note that in the region $\mathcal{A}_{3,1}$ we have
		$$
		\left|(1-a)\xi_2-\xi\right|=\left|\frac{1}{2}\xi_2-\xi +\left(\frac{1}{2}-a\right)\xi_2\right|\geq \left(a-\frac{1}{2}\right)\left|\xi_2\right|-\frac{2a-1}{4}\left|\xi_2\right|\geq c |\xi_2|.
		$$
		
		Then this case  is similar to the case $\mathcal{A}_{2}$. This concludes the proof that $J_1$ is limited.

		To prove that $J_2$ is bounded we just observe that 
		\begin{eqnarray*}
			\dfrac{1}{\langle\tau+a\xi^2\rangle^{2b}}\displaystyle\int_{\mathbb{R}}\dfrac{\langle\xi_2\rangle^{-2s+2|{\kappa}|}\chi_{\mathcal{R}_2}}{\langle\tau_2+2\xi^2+\xi_2^2-2\xi\xi_2\rangle^{2b+2d-1}}d\xi &= &  \dfrac{1}{\langle\omega_2 \rangle^{2b}}\displaystyle\int_{\mathbb{R}}\dfrac{\langle\xi_2\rangle^{-2s+2|{\kappa}|}\chi_{\mathcal{A}_{3,2}}}{\langle\tau_2+2\xi^2+\xi_2^2-2\xi\xi_2\rangle^{2b+2d-1}}d\xi\\
			&\leq & \dfrac{c}{\langle\omega_2 \rangle^{2b}}\displaystyle\int_{\langle\eta\rangle\leq 4 \langle\omega_2\rangle}\dfrac{\langle\xi_2\rangle^{-2s+2|{\kappa}|-1}}{\langle\eta\rangle^{2b+2d-1}}d\eta\\
			&\leq & c {\langle\omega_2 \rangle^{2-4b-2d}}\leq c.
		\end{eqnarray*}
		Where in the first inequality above, we made the change of variable $\eta=\tau_2+2\xi^2+\xi_2^2-2\xi\xi_2$ and used the fact that 
		$$
		|\eta|=|\omega_2+ (\omega-\omega_1-\omega_2)|\leq 4 |\omega_2|.
		$$
		Finally, we estimate $J_3$. 
		Analogously to the last estimate we get 
		\begin{eqnarray*}
			\dfrac{1}{\langle\tau-\xi_1^2\rangle^{2b}}\displaystyle\int_{\mathbb{R}}\dfrac{\langle\xi_2\rangle^{-2s+2|{\kappa}|}\chi_{\mathcal{R}_3}}{\langle\tau_1-a\xi_2^2+\xi^2\rangle^{2b+2d-1}}d\xi_2 &= &\dfrac{1}{\langle\omega_1\rangle^{2b}}\displaystyle\int_{|\xi_2|>1}\dfrac{\langle\xi_2\rangle^{-2s+2|{\kappa}|}\chi_{\mathcal{A}_{3,3}}}{\langle\tau_1-a\xi_2^2+\xi^2\rangle^{2b+2d-1}}d\xi_2 \\
			&\leq & \dfrac{c}{\langle\omega_1\rangle^{2b}}\displaystyle\int_{\langle\eta\rangle\leq 4 \langle\omega_1\rangle}\dfrac{\langle\xi_2\rangle^{-2s+2|{\kappa}|-1}}{\langle\eta\rangle^{2b+2d-1}}d\eta\\
			&\leq & c {\langle\omega_1 \rangle^{2-4b-2d}}\leq c.
		\end{eqnarray*}
	Where we have used the fact 
		$\tau_1-a\xi_2^2+\xi^2=\omega_1+(\omega -\omega_1-\omega_2)$.
		
		This finishes the proof of the case $a>\frac12$.

	\medskip
	
	  \item Case $a=\frac12$ and $|\kappa|\leq s$:
	  
	  	In this case, we do not have a nice dispersion relation. Then, we consider 
	  	$\mathcal{R}_1=\mathbb{R}^4$ and $\mathcal{R}_2=\mathcal{R}_3=\varnothing $. Thus, we only need to prove that $J_1$ is bounded.
	  	If $|\kappa|\leq s$ then  $$\dfrac{1}{\langle\omega\rangle^{2d}}\displaystyle\int_{|\xi_2|\geq 1}\dfrac{1}{\langle\tau-\frac12\xi_2^2-2\xi\xi_2+\xi^2\rangle^{4b-1}}d\xi_2\leq c,$$ 
	  	since $b>3/8$ and $d>0$. This finishes the proof of the Lemma \ref{l1}.
\end{itemize}

\end{proof}

\begin{proof}[Proof of the Lemma \ref{l2}]
The inequality \ref{desigualdade2}  produces the following dispersion relation
	\begin{equation}\label{relacao2}
	\left\lbrace 
	\begin{array}{lll}
	\tau=\tau_1+\tau_2 & \xi=\xi_1+\xi_2&     \\
	\lambda=\tau+a\xi^2, & \lambda_1=\tau_1+\xi_1^2,& \lambda_2=\tau_2+\xi_2^2
	\end{array}
	\right.
	\end{equation}
	
 \begin{itemize}
	\item $0<a<\frac12$ $(\sigma >2)$ and $s<\min\left\lbrace \kappa +1,\  \ 2\kappa +1\right\rbrace$.
	
		Analogously to the previous lemma, the estimate (\ref{desigualdade2}) is equivalent
		to prove that the following integral functions are  are bounded
			\begin{enumerate}
			\item[] \begin{equation}\label{j4} J_4(\xi,\tau)= \dfrac{1}{\langle\lambda\rangle^{2d}}\displaystyle\int_{\mathbb{R}}\dfrac{\langle \xi\rangle^{2s}\langle \xi_1\rangle^{-2{\kappa}} \langle \xi_2\rangle^{-2{\kappa}}\chi_{\mathcal{S}_1}}{\langle\tau+2\xi_2^2-2\xi\xi_2+\xi^2\rangle^{{4b-1}}}d\xi_2;\end{equation}
			\item[]\begin{equation}\label{j5} J_5(\xi_2,\tau_2)= \dfrac{1}{\langle\lambda_2\rangle^{2b}}\displaystyle\int_{\mathbb{R}}\dfrac{\langle \xi\rangle^{2s}\langle \xi_1\rangle^{-2{\kappa}} \langle \xi_2\rangle^{-2{\kappa}}\chi_{\mathcal{S}_2}}{\langle\tau_2+(a-1)\xi^2-\xi_2^2+2\xi\xi_2\rangle^{{2b+2d-1}}}d\xi;\end{equation}
			\item[] \begin{equation}\label{j6} J_6(\xi,\tau)= \dfrac{1}{\langle\lambda_1\rangle^{2b}}\displaystyle\int_{\mathbb{R}}\dfrac{\langle \xi\rangle^{2s}\langle \xi_1\rangle^{-2{\kappa}} \langle \xi_2\rangle^{-2{\kappa}}\chi_{\mathcal{S}_3}}{\langle\tau_1+a\xi^2+\xi_2^2\rangle^{{2b+2d-1}}}d\xi_2,\end{equation}
		\end{enumerate}
	 where  $\mathcal{S}_1\cup\mathcal{S}_2\cup\mathcal{S}_3=\mathbb{R}^4$ with $\mathcal{S}_j$ measurable.
		
		Note that 
		\begin{eqnarray*}
			|\lambda-\lambda_1-\lambda_2|&=& |a\xi^2-\xi_1^2-\xi_2^2|\\
			&\geq & |1-a|(\xi_1^2+\xi_2^2)-2a|\xi_1\xi_2|, \ \ \mbox{ (since } 0<a<\frac{1}{2})\\
			&\geq & (1-a)(\xi_1^2+\xi_2^2)-a(\xi_1^2+\xi_2^2)=(1-2a)(\xi_1^2+\xi_2^2),
		\end{eqnarray*}
		where we have used that $\xi=\xi_1+\xi_2$ and $|\xi|\leq |\xi_1|+|\xi_2|\leq 2\max \{\xi_1,\xi_2\}.$ 
		Then, $$3\max \{|\lambda|,|\lambda_1|, |\lambda_2|\}\geq (1-2a)\max\{\xi_1^2,\xi_2^2\}\geq \frac{1-2a}{4} \xi^2.$$
		Therefore, supposing that $|\xi|\geq 1$, we have
		$$\dfrac{1}{\max \{|\lambda|,|\lambda_1|, |\lambda_2|\}}\leq  \dfrac{c}{|\xi|^2}.$$
		Now, we define the regions $\mathcal{S}_i$. 
		\begin{equation}
			\mathcal{S}_1= \bigg\{ |\xi|\geq 1, |\lambda|=\max \{|\lambda|,|\lambda_1|, |\lambda_2|\}\bigg\}\cup\bigg\{|\xi| \leq 1\bigg\}\subset \mathbb{R}^4_{\xi,\tau,\xi_2,\tau_2},
		\end{equation}
		\begin{equation}
			\mathcal{S}_2= \bigg\{ |\xi|\geq 1, |\lambda_2|=\max \{|\lambda|,|\lambda_1|, |\lambda_2|\}\bigg\}\subset \mathbb{R}^4_{\xi,\tau,\xi_2,\tau_2}\ \text{and}
		\end{equation}
		\begin{equation}
			\mathcal{S}_3= \bigg\{ |\xi|\geq 1, |\lambda_1|=\max \{|\lambda|,|\lambda_1|, |\lambda_2|\}\bigg\}\subset \mathbb{R}^4_{\xi,\tau,\xi_2,\tau_2}.
		\end{equation}
		
		For $\kappa\geq 0$, we have 
		$\langle \xi_1 \rangle^{-2\kappa} \langle \xi_2 \rangle^{-2\kappa}\leq \langle \xi \rangle^{-2\kappa}$ and in this case 
		\begin{equation}
			J_4 \ \leq \displaystyle\int_{\mathbb{R}}\dfrac{\langle \xi\rangle^{2s-2{\kappa}+4d}\chi_{\mathcal{S}_1}}{\langle\tau+\xi_2^2-2\xi\xi_2+\xi^2\rangle^{{4b-1}}}d\xi_2.
		\end{equation}  
		Therefore, $J_4$ is bounded since $s-{\kappa}+2d<0$ for $s-{\kappa}<2d$.\\
		Note that $J_5$ and $J_6$ satisfies,
		\begin{equation}
			J_5 \ \leq \displaystyle\int_{\mathbb{R}}\dfrac{\langle \xi\rangle^{2s-2{\kappa}-4b}\chi_{\mathcal{S}_2}}{\langle\tau_2+(a-1)\xi^2-\xi_2^2+2\xi\xi_2\rangle^{{2b+2d-1}}}d\xi
		\end{equation}  
		and
		\begin{equation}
			J_6 \ \leq \displaystyle\int_{\mathbb{R}}\dfrac{\langle \xi\rangle^{2s-2{\kappa}-4b}\chi_{\mathcal{S}_3}}{\langle\tau_1+a\xi^2+\xi_2^2\rangle^{{2b+2d-1}}}d\xi_2,
		\end{equation}
		and that they are bounded since $s-{\kappa}<2b$ and $2b+2d-1> \frac{1}{2}$, that is, $\frac{3}{8}  <b,d<\frac{1}{2}.$

		The case $\kappa<0$ we need to separate in sub-cases as follows: 
		\begin{enumerate}
			\item Considering $|\xi_1|\leq \frac{2}{3}|\xi_2|$: we have $\langle\xi_1\rangle^{-2\kappa}\langle\xi_2\rangle^{-2\kappa}\leq \langle \xi_2\rangle^{-4\kappa}$. Moreover, $|\xi_2|\leq |\xi_1|+|\xi|\leq \frac{2|\xi_2|}{3}+|\xi|$, hence $|\xi_2|\leq 3|\xi|$. Therefore, $$\langle \xi\rangle^{2s}\langle \xi_1\rangle^{-2\kappa} \langle \xi_2\rangle^{-2\kappa}\leq \langle \xi\rangle^{2s-4\kappa}.$$
			
			\item Supposing  $|\xi_2|\leq \frac{2}{3}|\xi_1|$, we have,  $$\langle \xi\rangle^{2s}\langle \xi_1\rangle^{-2\kappa} \langle \xi_2\rangle^{-2\kappa}\leq \langle \xi\rangle^{2s-4\kappa}.$$
			
			\item The last case, $\frac{2}{3}|\xi_2|<|\xi_1|<\frac{3}{2}|\xi_2|$.
			\begin{enumerate}
				\item If $\xi_1,\ \xi_2\geq 0$ then $ \frac{2}{3}\xi_2<\xi_1<\frac{3}{2}\xi_2\Longrightarrow \frac{5}{3}\xi_2<\xi <\frac{5}{2}\xi_2$. Hence, $$\langle \xi\rangle^{2s}\langle \xi_1\rangle^{-2{\kappa} } \langle \xi_2\rangle^{-2{\kappa} }\leq \langle \xi\rangle^{2s-4\kappa}.$$
				\item If $\xi_1,\ \xi_2\leq 0$ then $ \frac{-2}{3}\xi_2<-\xi_1<\frac{-3}{2}\xi_2\Longrightarrow \frac{-5}{3}\xi_2<-\xi <\frac{-5}{2}\xi_2$, thus $|\xi_2|<\frac{3}{5}|\xi|$. Hence, $$\langle \xi\rangle^{2s}\langle \xi_1\rangle^{-2{\kappa} } \langle \xi_2\rangle^{-2{\kappa} }\leq \langle \xi\rangle^{2s-4\kappa}.$$
				\item If $\xi_1>0$ and $\xi_2<0$ then $\frac{-2}{3}\xi_2<\xi_1<\frac{-3}{2}\xi_2\Longrightarrow \frac{1}{3}\xi_2<\xi <\frac{-1}{2}\xi_2\Longrightarrow |\xi|<\frac{1}{2}|\xi_2|$.
				\item If $\xi_1<0$ and $\xi_2>0$ then $\frac{2}{3}\xi_2<-\xi_1<\frac{3}{2}\xi_2\Longrightarrow \frac{-1}{3}\xi_2<-\xi <\frac{1}{2}\xi_2$, consequently $|\xi|<\frac{1}{2}|\xi_2|$.
			\end{enumerate}
			
		\end{enumerate}

		The cases (1), (2), (3.a) and (3.b) are true  for $\kappa<0 $ and $s<2{\kappa}+1$.
		
		Indeed, given $\mathcal{A}\subset \mathbb{R}^4$ the set of the elements of $\mathbb{R}^4$ that satisfies one of the conditions (1), (2), (3.a) or (3.b), given $\mathcal{B}=\mathbb{R}^4 \setminus \mathcal{A}$.
		Now consider 
		$\mathcal{A}_{i}=\mathcal{S}_{i}\cap \mathcal{A}$ and $\mathcal{B}_{i}=\mathcal{S}_{i}\cap \mathcal{B}.$

		Analyzing the restrictions $\mathcal{A}_{i}$, we get:
		\begin{eqnarray*}
			J_4 &&= \dfrac{1}{\langle\lambda\rangle^{2d}}\displaystyle\int_{\mathbb{R}}\dfrac{\langle \xi\rangle^{2s}\langle \xi_1\rangle^{-2{\kappa} } \langle \xi_2\rangle^{ -2{\kappa} }\chi_{\mathcal{A}_{1}}}{\langle\tau+2\xi_2^2-2\xi\xi_2+\xi^2\rangle^{{4b-1}}}d\xi_2\\
			&&\leq \displaystyle\int_{\mathbb{R}}\dfrac{\langle \xi\rangle^{2s-4\kappa-4d} \chi_{\mathcal{A}_{1}}}{\langle\tau+2\xi_2^2-2\xi\xi_2+\xi^2\rangle^{{4b-1}}}d\xi_2.
		\end{eqnarray*}
		Then $J_4$ is bounded for $s\leq 2{\kappa}+2d$ and $b>3/8$.
		\begin{eqnarray*}
			J_5 &&= \dfrac{1}{\langle\lambda_2\rangle^{2b}}\displaystyle\int_{\mathbb{R}}\dfrac{\langle \xi\rangle^{2s}\langle \xi_1\rangle^{-2{\kappa} } \langle \xi_2\rangle^{-2{\kappa} }\chi_{\mathcal{A}_{2}}}{\langle\tau_2+(a-1)\xi^2-\xi_2^2+2\xi\xi_2\rangle^{{2b+2d-1}}}d\xi\\
			&&\leq \displaystyle\int_{\mathbb{R}}\dfrac{\langle \xi\rangle^{2s-4\kappa-4b}\chi_{\mathcal{A}_{2}}}{\langle\tau_2+(a-1)\xi^2-\xi_2^2+2\xi\xi_2\rangle^{{2b+2d-1}}}d\xi \ \ 
		\end{eqnarray*}
		and $J_5$ is bounded for $s\leq 2{\kappa}+2b$ and $1/2>b$.
		\begin{eqnarray*}
			J_6 &&= \dfrac{1}{\langle\lambda_1\rangle^{2b}}\displaystyle\int_{\mathbb{R}}\dfrac{\langle \xi\rangle^{2s}\langle \xi_1\rangle^{-2{\kappa} } \langle \xi_2\rangle^{-2{\kappa} }\chi_{\mathcal{A}_{3}}}{\langle\tau_1+a\xi^2+\xi_2^2\rangle^{{2b+2d-1}}}d\xi_2\\
			&&\leq \displaystyle\int_{\mathbb{R}}\dfrac{\langle \xi\rangle^{2s-4\kappa-4b}\chi_{\mathcal{A}_{3}}}{\langle\tau_1+a\xi^2+\xi_2^2\rangle^{{2b+2d-1}}}d\xi_2.
		\end{eqnarray*}
		Then $J_6$ is also bounded for $s\leq 2{\kappa}+2b$.

		To analyze the remaining cases (which is equivalent to supposing $|\xi|<\frac{1}{2}|\xi_2|$ and $|\xi_1|\cong |\xi_2|$) let us consider
		them as regions $\mathcal{B}_{i}$: 
		
		We start by estimating $J_4$.
		\begin{eqnarray*}
			J_4 &&= \dfrac{1}{\langle\lambda\rangle^{2d}}\displaystyle\int_{\mathbb{R}}\dfrac{\langle \xi\rangle^{2s}\langle \xi_1\rangle^{-2{\kappa} } \langle \xi_2\rangle^{-2{\kappa} }\chi_{\mathcal{B}_{1}}}{\langle\tau+2\xi_2^2-2\xi\xi_2+\xi^2\rangle^{4b-1}}d\xi_2\\
			&&\leq \displaystyle\int_{\mathbb{R}}\dfrac{\langle \xi\rangle^{2s-4d}\langle \xi_1\rangle^{-4\kappa} \chi_{\mathcal{B}_{1}}}{\langle\tau+2\xi_2^2-2\xi\xi_2+\xi^2\rangle^{4b-1}}d\xi_2\\
			&& \leq \displaystyle\int_{\mathbb{R}}\dfrac{\langle \xi\rangle^{2s-4d}\langle \xi_1\rangle^{-4\kappa} \chi_{\mathcal{B}_{1}}}{\langle\eta \rangle^{4b-1/2}}d\eta
		\end{eqnarray*}
		We use above, $|2\xi_2-\xi|\cong \langle\eta \rangle^{1/2}$.
		
		Hence,
		$J_4\leq \langle \xi\rangle^{2s-4d}\langle \xi_1\rangle^{-4\kappa}\displaystyle\int_{\mathbb{R}}\dfrac{d\eta}{\langle\eta \rangle^{4b-1/2}}$, that is bounded because $4b-1/2>1$ and $2s\leq 4\kappa+4d$.
		$$\langle \xi\rangle^{2s-4d}\langle \xi_1\rangle^{-4\kappa}\leq \langle \xi \rangle^{2s-4\kappa-4d}\leq c.$$
		We continue to estimate $J_5$:
		\begin{eqnarray*}
			J_5 &&= \dfrac{1}{\langle\lambda_2\rangle^{2b}}\displaystyle\int_{\mathbb{R}}\dfrac{\langle \xi\rangle^{2s}\langle \xi_1\rangle^{-2{\kappa} } \langle \xi_2\rangle^{-2{\kappa} }\chi_{\mathcal{B}_{2}}}{\langle\tau_2+(a-1)\xi^2-\xi_2^2+2\xi\xi_2\rangle^{2b+2d-1}}d\xi\\
			&&\leq  \dfrac{1}{\langle\lambda_2\rangle^{2b}}\displaystyle\int_{\mathbb{R}}\dfrac{\langle \xi\rangle^{2s}\langle \xi_2\rangle^{-4\kappa}\chi_{\mathcal{B}_{2}}}{\langle\tau_2+(a-1)\xi^2-\xi_2^2+2\xi\xi_2\rangle^{2b+2d-1}}d\xi
		\end{eqnarray*}
		Setting $\eta=\tau_2+(a-1)\xi^2-\xi_2^2+2\xi\xi_2$, such that $d\eta=2(\xi_2+(a-1)\xi)d\xi$. Now, as $0<a<\frac{1}{2},$ it follows $|a-1|< 1$ and therefore $|\xi_2+(a-1)\xi|\geq \frac{1}{2}|\xi_2|\cong \langle \eta\rangle^{1/2}$. 
		Thus,
		\begin{eqnarray*}
			J_5 && \leq  \dfrac{1}{\langle\lambda_2\rangle^{2b}}\displaystyle\int_{\mathbb{R}}\dfrac{\langle \xi\rangle^{2s}\langle \xi_2\rangle^{-4\kappa}}{\langle \eta\rangle^{2b+2d-1/2}}d\eta \\
			&&\leq \dfrac{1}{\langle\lambda_2\rangle^{2b}}\displaystyle\int_{\mathbb{R}}\dfrac{\langle \xi_2\rangle^{\max\{0,2s\}-4\kappa}}{\langle\eta \rangle^{2b+2d-1/2}}d\eta \leq \langle \xi_2\rangle^{\max\{0,2s\}-4\kappa-4d}.			
		\end{eqnarray*}

		The estimate of $J_6$ is analogous of the estimate of $J_5$.

	\item $ a>\frac12$ $(0<\sigma<2)$ and $s< \min\left\lbrace \kappa +1/2,\  \ 2\kappa +1/2\right\rbrace $.
	
		The dispersion relation above is zero in two straight lines.
		
		Now, we define, 
		\begin{eqnarray*}
			\mathcal{B}_1&=& \{ |\xi| \leq 1\}\subset \mathbb{R}^4,  \\
			\mathcal{B}_2&=& \left\{ |\xi|\geq 1, \left|\xi_2-\frac{1}{2}\xi\right|>\frac{2a-1}{4}|\xi|\right\}\subset \mathbb{R}^4\ \text{and} \\
			\mathcal{B}_3&=& \left\{ |\xi|\geq 1, \left|(1-a)\xi-\xi_2\right|>\frac{2a-1}{4}|\xi|\right\}\subset \mathbb{R}^4.  
		\end{eqnarray*}
	
It follows that $\mathbb{R}^4= \mathcal{B}_1\cup \mathcal{B}_2\cup \mathcal{B}_3.$
		
		Now, we set the following sets
		\begin{eqnarray*}
			\mathcal{B}_{3,1}&=& \mathcal{B}_3 \cap \left\{ |\lambda|\geq \max \{|\lambda_1|, |\lambda_2|\} \right\}, \\
			\mathcal{B}_{3,2}&=& \mathcal{B}_3 \cap \left\{ |\lambda_2|\geq \max \{|\lambda_1|, |\lambda|\}\right\}\ \text{and}\\
			\mathcal{B}_{3,3}&=& \mathcal{B}_3\cap \left\{ |\lambda_1|\geq \max \{|\lambda|, |\lambda_2|\}\right\}.
		\end{eqnarray*}
		We define the regions $\mathcal{S}_i$ (analogously to the proof of case $a<\frac12$), 
		setting $\mathcal{S}_1=\mathcal{B}_1\cup\mathcal{B}_2\cup \mathcal{B}_{3,1}$, $\mathcal{S}_2=\mathcal{B}_{3,2}$ and $\mathcal{S}_3=\mathcal{B}_{3,3}.$
		
		For $\kappa\geq 0$, we have
		$\langle \xi_1 \rangle^{-2{\kappa} } \langle \xi_2 \rangle^{-2{\kappa} }\leq \langle \xi \rangle^{-2{\kappa} }$:
		
		\begin{equation}\label{j41}
			J_4 \leq \dfrac{1}{\langle\lambda\rangle^{2d}}\displaystyle\int_{\mathbb{R}}\dfrac{\langle \xi\rangle^{2s-2{\kappa} } \chi_{\mathcal{S}_1}}{\langle\tau+2\xi_2^2-2\xi\xi_2+\xi^2\rangle^{4b-1}}d\xi_2,
		\end{equation}
		
		\begin{equation}\label{j51}
			J_5 \leq \dfrac{1}{\langle\lambda_2\rangle^{2b}}\displaystyle\int_{\mathbb{R}}\dfrac{\langle \xi\rangle^{2s-2{\kappa} }\chi_{\mathcal{S}_2}}{\langle\tau_2+(a-1)\xi^2-\xi_2^2+2\xi\xi_2\rangle^{2b+2d-1}}d\xi\ \text{and}
		\end{equation}
		
		\begin{equation}\label{j61}
			J_6 \leq \dfrac{1}{\langle\lambda_1\rangle^{2b}}\displaystyle\int_{\mathbb{R}}\dfrac{\langle \xi\rangle^{2s-2{\kappa} }\chi_{\mathcal{S}_3}}{\langle\tau_1+a\xi^2+\xi_2^2\rangle^{2b+2d-1}}d\xi_2. 
		\end{equation}
		To complete the proof that $J_4$ is bounded it is it is suffices that  (\ref{j41}) satisfies:  
		\begin{equation*}
			\dfrac{1}{\langle\lambda\rangle^{2d}}\displaystyle\int_{\mathbb{R}}\dfrac{\langle \xi\rangle^{2s-2{\kappa} } \chi_{\mathcal{B}_1}}{\langle\tau+2\xi_2^2-2\xi\xi_2+\xi^2\rangle^{4b-1}}d\xi_2\leq \dfrac{c}{\langle\lambda\rangle^{2d}}\displaystyle\int_{\mathbb{R}}\dfrac{1}{\langle\tau+2\xi_2^2-2\xi\xi_2+\xi^2\rangle^{4b-1}}d\xi_2\leq c;
		\end{equation*}
		
		\begin{equation*} 
			\dfrac{1}{\langle\lambda\rangle^{2d}}\displaystyle\int_{\mathbb{R}}\dfrac{\langle \xi\rangle^{2s-2{\kappa} } \chi_{\mathcal{B}_2}}{\langle\tau+2\xi_2^2-2\xi\xi_2+\xi^2\rangle^{4b-1}}d\xi_2\leq \dfrac{c}{\langle\lambda\rangle^{2d}}\displaystyle\int_{{\langle\eta\rangle\leq 4\langle\lambda\rangle}}\dfrac{\langle \xi\rangle^{2s-2{\kappa}-1}}{\langle\eta\rangle^{4b-1}}d\xi_2\leq \dfrac{\langle \xi\rangle^{2s-2{\kappa}-1}}{\langle\lambda\rangle^{4b+2d-2}}   \mbox{ and}
		\end{equation*}
		
		\begin{equation*}
			\dfrac{1}{\langle\lambda\rangle^{2d}}\displaystyle\int_{\mathbb{R}}\dfrac{\langle \xi\rangle^{2s-2{\kappa} } \chi_{\mathcal{B}_{3,1}}}{\langle\tau+2\xi_2^2-2\xi\xi_2+\xi^2\rangle^{4b-1}}d\xi_2\leq\dfrac{c}{\langle\lambda\rangle^{2d}}\displaystyle\int_{{\langle\eta\rangle\leq 4\langle\lambda\rangle}}\dfrac{\langle \xi\rangle^{2s-2{\kappa}-1}}{\langle\eta\rangle^{4b-1}}d\xi_2\leq c\dfrac{\langle \xi\rangle^{2s-2{\kappa}-1}}{\langle\lambda\rangle^{4b+2d-2}}.
		\end{equation*}
		In the estimates above, we used the fact $b,d>3/8$ and $s\leq \kappa+ 1/2$.

		Let us estimate (\ref{j51}), using the fact that $$\eta=\tau_2+(a-1)\xi^2-\xi_2^2+2\xi\xi_2=\lambda_2+(\lambda-\lambda_1-\lambda_2),$$ which give us $d\eta=2((1-a)\xi-\xi_2)d\xi$, so
		\begin{eqnarray*}
			\dfrac{1}{\langle\lambda_2\rangle^{2b}}\displaystyle\int_{\mathbb{R}}\dfrac{\langle \xi\rangle^{2s-2{\kappa} }\chi_{\mathcal{B}_{3,2}}}{\langle\tau_2+(a-1)\xi^2-\xi_2^2+2\xi\xi_2\rangle^{2b+2d-1}}d\xi & \leq & \dfrac{c}{\langle\lambda_2\rangle^{2b}}\displaystyle\int_{\langle\eta\rangle\leq 4\langle\lambda_2\rangle}\dfrac{\langle \xi\rangle^{2s-2{\kappa}-1}}{\langle\eta\rangle^{2b+2d-1}}d\eta \\
			& \leq & \dfrac{c}{\langle\lambda_2\rangle^{4b+2d-2}}\leq c.
		\end{eqnarray*}
		Now let us estimate (\ref{j61}). This is completely analogous to the previous estimate.
		\begin{eqnarray*}
			\dfrac{1}{\langle\lambda_1\rangle^{2b}}\displaystyle\int_{\mathbb{R}}\dfrac{\langle \xi\rangle^{2s-2{\kappa} }\chi_{\mathcal{S}_3}}{\langle\tau_1+a\xi^2+\xi_2^2\rangle^{2b+2d-1}}d\xi_2 & \leq & \dfrac{c}{\langle\lambda_1\rangle^{2b}}\displaystyle\int_{\langle\eta\rangle\leq 4\langle\lambda_1\rangle}\dfrac{\langle \xi\rangle^{2s-2{\kappa}-1}}{\langle\eta\rangle^{2b+2d-1}}d\eta \\
			& \leq & \dfrac{c}{\langle\lambda_1\rangle^{4b+2d-2}}\leq c.
		\end{eqnarray*}
		This concludes the case $\kappa \geq 0$.
		
		\vspace{0.5cm}
		The case $\kappa<0$ will be separated into sub-cases: 
		
		\begin{enumerate}
			\item Supposing $|\xi_1|\leq \frac{2}{3}|\xi_2|$, then, $\langle\xi_1\rangle^{-2{\kappa} }\langle\xi_2\rangle^{-2{\kappa} }\leq \langle \xi_2\rangle^{-4\kappa}$. Moreover,$|\xi_2|\leq |\xi_1|+|\xi|\leq \frac{2|\xi_2|}{3}+|\xi|$, hence $|\xi_2|\leq 3|\xi|$. Therefore, $$\langle \xi\rangle^{2s}\langle \xi_1\rangle^{-2{\kappa} } \langle \xi_2\rangle^{-2{\kappa} }\leq \langle \xi\rangle^{2s-4\kappa}.$$
			
			\item Supposing  $|\xi_2|\leq \frac{2}{3}|\xi_1|$ we have the same result, that is,  $$\langle \xi\rangle^{2s}\langle \xi_1\rangle^{-2{\kappa} } \langle \xi_2\rangle^{-2{\kappa} }\leq \langle \xi\rangle^{2s-4\kappa}.$$
			
			\item For the case, $\frac{2}{3}|\xi_2|<|\xi_1|<\frac{3}{2}|\xi_2|$, we need to do the following:
			\begin{enumerate}
				\item If $\xi_1,\ \xi_2\geq 0$ then $ \frac{2}{3}\xi_2<\xi_1<\frac{3}{2}\xi_2\Longrightarrow \frac{5}{3}\xi_2<\xi <\frac{5}{2}\xi_2$. Hence, $$\langle \xi\rangle^{2s}\langle \xi_1\rangle^{-2{\kappa} } \langle \xi_2\rangle^{-2{\kappa} }\leq \langle \xi\rangle^{2s-4\kappa}.$$
				\item If $\xi_1,\ \xi_2\leq 0$ then $ \frac{-2}{3}\xi_2<-\xi_1<\frac{-3}{2}\xi_2\Longrightarrow \frac{-5}{3}\xi_2<-\xi <\frac{-5}{2}\xi_2$, so $|\xi_2|<\frac{3}{5}|\xi|$. Hence, $$\langle \xi\rangle^{2s}\langle \xi_1\rangle^{-2{\kappa} } \langle \xi_2\rangle^{-2{\kappa} }\leq \langle \xi\rangle^{2s-4\kappa}.$$
				\item If $\xi_1>0$ and $\xi_2<0$ then $\frac{-2}{3}\xi_2<\xi_1<\frac{-3}{2}\xi_2\Longrightarrow \frac{1}{3}\xi_2<\xi <\frac{-1}{2}\xi_2$, now $|\xi|<\frac{1}{2}|\xi_2|$.
				\item If $\xi_1<0$ and $\xi_2>0$ then $\frac{2}{3}\xi_2<-\xi_1<\frac{3}{2}\xi_2\Longrightarrow \frac{-1}{3}\xi_2<-\xi <\frac{1}{2}\xi_2$, which give us $|\xi|<\frac{1}{2}|\xi_2|$.
			\end{enumerate}
			
		\end{enumerate}
		
		The cases (1), (2), (3.a) and (3.b) are valid for $\kappa<0 $ and $s\leq2{\kappa}+\frac{1}{2}$.
		
		Indeed, let $\mathcal{C}\subset \mathbb{R}^4$ be the set of element $\mathbb{R}^4$ that satisfies one of the conditions (1), (2), (3.a) or (3.b).
		Now consider 
		$\mathcal{C}_{i}=\mathcal{S}_{i}\cap \mathcal{C}$.

		Analyzing the restrictions on $\mathcal{C}_{i}$, we get:
		\begin{eqnarray*}
			\dfrac{1}{\langle\lambda\rangle^{2d}}\displaystyle\int_{\mathbb{R}}\dfrac{\langle \xi\rangle^{2s}\langle \xi_1\rangle^{-2{\kappa} } \langle \xi_2\rangle^{-2{\kappa} }\chi_{\mathcal{C}_{1}}}{\langle\tau+\xi_2^2-2\xi\xi_2+\xi^2\rangle^{4b-1}}d\xi_2 & \leq & \dfrac{1}{\langle\lambda\rangle^{2d}}\displaystyle\int_{\mathbb{R}}\dfrac{\langle \xi\rangle^{2s-4\kappa} \chi_{\mathcal{C}_{1}}}{\langle\tau+\xi_2^2-2\xi\xi_2+\xi^2\rangle^{4b-1}}d\xi_2\\
			& \leq &\dfrac{c}{\langle\lambda\rangle^{2d}}\displaystyle\int_{\langle\eta\rangle\leq 4\langle\lambda\rangle}\dfrac{\langle \xi\rangle^{2s-4\kappa-1} }{\langle\eta\rangle^{4b-1}}d\xi_2 \leq  c, \end{eqnarray*}
	since $3/8<b,d \mbox{ and } s\leq 2{\kappa}+1/2$.

		\begin{eqnarray*}
			\dfrac{1}{\langle\lambda_2\rangle^{2b}}\displaystyle\int_{\mathbb{R}}\dfrac{\langle \xi\rangle^{2s}\langle \xi_1\rangle^{-2{\kappa} } \langle \xi_2\rangle^{-2{\kappa} }\chi_{\mathcal{C}_{2}}}{\langle\tau_2+(a-1)\xi^2-\xi_2^2+2\xi\xi_2\rangle^{2b+2d-1}}d\xi
			&&\leq \dfrac{c}{\langle\lambda_2\rangle^{2b}}\displaystyle\int_{\langle\eta\rangle\leq 4\langle\lambda_2\rangle}\dfrac{\langle \xi\rangle^{2s-4\kappa-1}}{\langle\eta\rangle^{2b+2d-1}}d\eta \\
			&&\leq \dfrac{c}{\langle\lambda_2\rangle^{4b+2d-2}}\leq c.
		\end{eqnarray*}

		\begin{eqnarray*}
			\dfrac{1}{\langle\lambda_1\rangle^{2b}}\displaystyle\int_{\mathbb{R}}\dfrac{\langle \xi\rangle^{2s}\langle \xi_1\rangle^{-2{\kappa} } \langle \xi_2\rangle^{-2{\kappa} }\chi_{\mathcal{C}_{3}}}{\langle\tau_1+a\xi^2+\xi_2^2\rangle^{2b+2d-1}}d\xi_2
			& &\leq \dfrac{c}{\langle\lambda_1\rangle^{2b}}\displaystyle\int_{\langle\eta\rangle\leq 4\langle\lambda_1\rangle}\dfrac{\langle \xi\rangle^{2s-4\kappa-1}}{\langle\eta\rangle^{2b+2d-1}}d\eta \\
			&&\leq \dfrac{c}{\langle\lambda_1\rangle^{4b+2d-2}}\leq c.
		\end{eqnarray*}
		
		Consider $\mathcal{D}=\mathbb{R}^4 \setminus \mathcal{C}$  and $\mathcal{D}_{i}=\mathcal{S}_{i}\cap \mathcal{D}.$
		To obtain the other cases (which is equivalent to supposing $|\xi|<\frac{1}{2}|\xi_2|$ and $|\xi_1|\sim |\xi_2|$) let us consider the regions $\mathcal{D}_{i}$: 
		
		We begin by estimating $J_4$.
		
		\begin{eqnarray*}
			\dfrac{1}{\langle\lambda\rangle^{2d}}\displaystyle\int_{\mathbb{R}}\dfrac{\langle \xi\rangle^{2s}\langle \xi_1\rangle^{-2{\kappa} } \langle \xi_2\rangle^{-2{\kappa} }\chi_{\mathcal{D}_{1}}}{\langle\tau+2\xi_2^2-2\xi\xi_2+\xi^2\rangle^{4b-1}}d\xi_2 & \leq  & \dfrac{1}{\langle\lambda\rangle^{2d}} \displaystyle\int_{\mathbb{R}}\dfrac{\langle \xi\rangle^{2s}\langle \xi_1\rangle^{-4\kappa} \chi_{\mathcal{D}_{1}}}{\langle\tau+2\xi_2^2-2\xi\xi_2+\xi^2\rangle^{4b-1}}d\xi_2\\
			& \leq & \dfrac{c}{\langle\lambda\rangle^{2d}}\displaystyle\int_{\langle\eta \rangle\leq 4 \langle\lambda \rangle}\dfrac{\langle \xi\rangle^{2s}\langle \xi_1\rangle^{-4{\kappa}-1} }{\langle\eta \rangle^{4b-1}}d\eta.
		\end{eqnarray*}
		Now, $|\xi_2-\xi|\geq |\xi_2|-|\xi|\geq \frac{1}{2}|\xi_2|\cong \frac{1}{2}|\xi_1|$.\\
		Hence,
		$J_4\leq \langle \xi\rangle^{2s}\langle \xi_1\rangle^{-4\kappa-1}\dfrac{c}{\langle\lambda\rangle^{2d}}\displaystyle\int_{\langle\eta \rangle\leq 4\langle\lambda \rangle}\dfrac{d\eta}{\langle\eta \rangle^{4b-1}}\leq c \dfrac{\langle \xi\rangle^{2s-4\kappa-1}}{\langle\lambda\rangle^{4b+2d-2}}$.

		Estimating $J_5$:
		\begin{eqnarray*}
			J_5 &&= \dfrac{1}{\langle\lambda_2\rangle^{2b}}\displaystyle\int_{\mathbb{R}}\dfrac{\langle \xi\rangle^{2s}\langle \xi_1\rangle^{-2{\kappa} } \langle \xi_2\rangle^{-2{\kappa} }\chi_{\mathcal{B}_{2}}}{\langle\tau_2+(a-1)\xi^2-\xi_2^2+2\xi\xi_2\rangle^{2b+2d-1}}d\xi\\
			&&\leq  \dfrac{1}{\langle\lambda_2\rangle^{2b}}\displaystyle\int_{\mathbb{R}}\dfrac{\langle \xi\rangle^{2s}\langle \xi_2\rangle^{-4\kappa}\chi_{\mathcal{B}_{2}}}{\langle\tau_2+(a-1)\xi^2-\xi_2^2+2\xi\xi_2\rangle^{2b+2d-1}}d\xi.
		\end{eqnarray*}
		Setting $\eta=\tau_2+(a-1)\xi^2-\xi_2^2+2\xi\xi_2$, which give $d\eta=2(\xi_2+(a-1)\xi)d\xi$. As $0<a<\frac{1}{2},$ we have $|a-1|\leq 1$ and therefore $|\xi_2+(a-1)\xi|\geq \frac{1}{2}|\xi_2|$.

		Hence,
		\begin{eqnarray*}
			J_5 && \leq  \dfrac{c}{\langle\lambda_2\rangle^{2b}}\displaystyle\int_{\langle\eta \rangle \leq c\langle\lambda_2\rangle}\dfrac{\langle \xi\rangle^{2s}\langle \xi_2\rangle^{-4\kappa-1}}{\langle \eta\rangle^{2b+2d-1}}d\eta \\
			&&\leq \dfrac{c}{\langle\lambda_2\rangle^{2b}}\displaystyle\int_{\langle\eta \rangle\leq c\langle\lambda_2\rangle}\dfrac{\langle \xi_2\rangle^{\max\{0,2s\}-4\kappa-1}}{\langle\eta \rangle^{2b+2d-1}}d\eta \\
			&& \leq c\langle \xi_2\rangle^{\max\{0,2s\}-4\kappa-1}\langle\lambda_2\rangle^{2-4b-2d}\leq c\langle \xi_2\rangle^{\max\{0,2s\}-4\kappa-2}.
		\end{eqnarray*}
		Since $2-4b-2d<0$, when $3/8<b,d$.

		Now, we estimate $J_6$.
		Remembering that 
		\begin{eqnarray*}
			J_6 &&= \dfrac{1}{\langle\lambda_1\rangle^{2b}}\displaystyle\int_{\mathbb{R}}\dfrac{\langle \xi\rangle^{2s}\langle \xi_1\rangle^{-2{\kappa} } \langle \xi_2\rangle^{-2{\kappa} }\chi_{\mathcal{B}_{3}}}{\langle\tau_1+a\xi^2+\xi_2^2\rangle^{2b+2d-1}}d\xi_2\\
			&& \leq \dfrac{1}{\langle\lambda_1\rangle^{2b}}\displaystyle\int_{\mathbb{R}}\dfrac{\langle \xi\rangle^{2s}\langle \xi_1\rangle^{-4\kappa} \chi_{\mathcal{B}_{3}}}{\langle\tau_1+a\xi^2+\xi_2^2\rangle^{2b+2d-1}}d\xi_2.
		\end{eqnarray*}
		Using $\eta=\tau_1+a\xi^2+\xi_2^2$, which give $d\eta=2\xi_2d\xi_2$. Now, 
		\begin{eqnarray*}
			|\eta| &&=|\tau_1+a\xi^2+\xi_2^2|\\
			&&=|\lambda_1+(a\xi^2+\xi_2^2-\xi_1^2)|\\
			&&\leq c |\lambda_1|.
		\end{eqnarray*}
		By using the fact that $|\xi_1|\cong |\xi_2|$, we have
		\begin{eqnarray*}
			J_6 && \leq \dfrac{c}{\langle\lambda_1\rangle^{2b}}\displaystyle\int_{\langle\eta \rangle\leq c\langle\lambda_1\rangle}\dfrac{\langle \xi\rangle^{2s}\langle \xi_1\rangle^{-4\kappa} }{|\xi_1|\langle\eta \rangle^{2b+2d-1}}d\xi_2\\
			&&\leq  c\langle \xi_1\rangle^{\max\{0,2s\}-4\kappa-1}\langle\lambda_1\rangle^{2-2d-4b}\\
			&& \leq c \langle \xi_1\rangle^{\max\{0,2s\}-4\kappa-2}.
		\end{eqnarray*}

	\item $a=\frac12$ $(\sigma=2)$ e $0 \leq s\leq \kappa $ 
	
		As in the previous case, we cannot take advantage of the dispersion relation. So let us take $\mathcal{S}_1=\mathbb{R}^4$ and $\mathcal{S}_2=\mathcal{S}_3=\varnothing$. Note that it is enough to estimate $J_4$. Initially assume that $\kappa \geq 0$, so we get 
		$\langle \xi_1 \rangle^{-2{\kappa} } \langle \xi_2 \rangle^{-2{\kappa} }\leq \langle \xi \rangle^{-2{\kappa} }$:
		
		\begin{equation*}
			J_4 \leq \dfrac{1}{\langle\lambda\rangle^{2d}}\displaystyle\int_{\mathbb{R}}\dfrac{\langle \xi\rangle^{2s-2{\kappa} } \chi_{\mathcal{S}_1}}{\langle\tau+2\xi_2^2-2\xi\xi_2+\xi^2\rangle^{4b-1}}d\xi_2. 
		\end{equation*}
		Finally, since $s\leq \kappa$, $b>3/8$ and $d>0$, we conclude that $J_4$ is bounded.
\end{itemize}	
	
\end{proof}
\section{Proof of Theorem \ref{teo1}}\label{6}

Let $a>0$ fix and $( \kappa,s)$ satisfying the hypothesis of theorem.  Choose $b=b(\kappa, s)<\frac{1}{2}$ such that the nonlinear estimates of Lemmas \ref{l1} and \ref{l2} are valid.   Let $\tilde{u}_{0}, \tilde{v}_{0}, \tilde{f}$ and $\tilde{g}$ be extensions of $u_{0}, v_{0}, f$ and $g$ in all line $\R$ such that 
\[
\begin{array}{l}
\left\|\tilde{u}_{0}\right\|_{H^{\kappa}(\mathrm{R})} \leq c\left\|u_{0}\right\|_{H^{\kappa}\left(\mathbb{R}^{+}\right)},\left\|\tilde{v}_{0}\right\|_{H^{s}(\mathbb{R})} \leq c\left\|v_{0}\right\|_{H^{s}\left(\mathbb{R}^{+}\right)},\\
\|\tilde{f}\|_{H^{\frac{2 \kappa+1}{4}}(\mathbb{R})} \leq c\|f\|_{H^{\frac{2 \kappa+1}{4}}\left(\mathbb{R}^{+}\right)} \text {and }\ \|\tilde{g}\|_{H^{\frac{2s+1}{4}(\mathbb{R})}} \leq c\|g\|_{H} \frac{2s+1}{4}\left(\mathbb{R}^{+}\right).
\end{array}
\]
Using \eqref{2.2} and \eqref{forc} we need to obtain a fixed point, in appropriate functional space $Z(\kappa,s)$,  for the operator $\Lambda=\left(\Lambda_{1}, \Lambda_{2}\right),$ given by
\[
\begin{aligned}
\Lambda_{1}(u, v)=& \psi(t) U_1(t) \tilde{u}_{0}(x)+\psi(t) \mathcal{S}\left(\alpha \psi_{T} \overline{u} v\right)(x, t) +\psi(t) e^{-i \frac{\lambda_{1} \pi}{4}} \mathcal{L}^{\lambda_{1}} h_{1}(x, t) \text { and } \\
\Lambda_{2}(u, v)=& \psi(t) U_a(t) \tilde{v}_{0}(x)+\psi(t) \mathcal{S}_a\left(\gamma \psi_{T}\left(u^{2}\right)\right)(x, t) +\psi(t) \frac{ e^{-i \frac{\lambda_{1} \pi}{4}}}{\sqrt a} \mathcal{L}_{a}^{\lambda_{2}} h_{2}(x, t),
\end{aligned}
\]

where
\[
h_{1}(t)=\left.\left[\psi(t) \tilde{f}(t)-\left.\psi(t) e^{i t \partial_{x}^{2}} \tilde{u}_{0}\right|_{x=0}-\psi(t) \mathcal{S}\left( \psi_{T}\overline{u}v \right)(0, t)\right]\right|_{(0,+\infty)}
\]
and
\[
h_{2}(t)=\left.\left[\psi(t) \tilde{g}(t)-\left.\psi(t) e^{i at \partial_{x}^{2}} \tilde{v}_{0}\right|_{x=0}-\psi(t) \mathcal{S}_a\left( \psi_{T}u^{2} \right)(0, t)\right]\right|_{(0,+\infty)},
\]
where $\lambda_{1}=\lambda_{1}(s)$ and $\lambda_{2}=\lambda_{2}(\kappa)$ are fixed numbers satisfying 
\begin{equation}\label{ks}
\begin{split}&\max\left\{\kappa-\frac12,-1\right\}<\lambda_{1}<\min \left\{\frac{1}{2}, \kappa+\frac{1}{2}\right\}\ \text{and}\\
&\max\left\{s-\frac12,-1\right\}<\lambda_{2}<\min \left\{\frac{1}{2}, s+\frac{1}{2}\right\}.
\end{split}
\end{equation} Then, Lemmas \ref{trace} and \ref{estimate} are to be valid. Observe that as the indexes $\kappa$ and $s$ are contained in the interval $(-\frac23,1)$, then the choices for $\lambda_1$ and $\lambda_2$ satisfying \eqref{ks} is possible.

We consider $\Lambda$ in the Banach space $Z=Z(\kappa , s)=Z_{1} \times Z_{2},$ where
\[
\begin{array}{l}
Z_{1}=\mathcal{C}\left(\mathbb{R}_{t} ; H^{\kappa}\left(\mathbb{R}_{x}\right)\right) \cap \mathcal{C}\left(\mathbb{R}_{x} ; H^{\frac{2\kappa+1}{4}}\left(\mathbb{R}_{t}\right)\right) \cap X^{\kappa, b} 
\\
\quad \quad\quad \quad \quad \quad\text{and}
\\
Z_{2}=\mathcal{C}\left(\mathbb{R}_{t} ; H^{s}\left(\mathbb{R}_{x}\right)\right) \cap \mathcal{C}\left(\mathbb{R}_{x} ; H^{\frac{2s+1}{4}}\left(\mathbb{R}_{t}\right)\right) \cap X_a^{s, b}. 
\end{array}
\]

By using the estimates obtained on Lemmas \ref{l2.1}, \ref{estimate}, \ref{l1}, \ref{l2}, \ref{lw1} and \ref{lw2} we can obtain
\begin{equation}\label{c1}
\left\|\Lambda_{1}(u, v)\right\|_{Z_{1}} \leq  c\left(\left\|u_{0}\right\|_{H^{\kappa}\left(\mathbb{R}^{+}\right)}+\|f\|_{H^{\frac{2 \kappa+1}{4}}\left(\mathbb{R}^{+}\right)}+T^{\epsilon}\|u\|_{X^{\kappa, b}}\|v\|_{X_a^{s, b} }\right)
\end{equation}
and
\begin{equation}\label{c2}
\left\|\Lambda_{2}(u, v)\right\|_{Z_{2}} \leq  c\left(\left\|v_{0}\right\|_{H^{s}\left(\mathbb{R}^{+}\right)}+\|g\|_{H^{\frac{2 s+1}{4}}\left(\mathbb{R}^{+}\right)}+T^{\epsilon}\|u\|^2_{X^{\kappa, b}} \right).
\end{equation}
Similarly we have
\begin{equation}\label{c3}
\begin{aligned}
\left\|\Lambda\left(u_{1}, v_{1}\right)-\Lambda\left(u_{2}, v_{2}\right)\right\|_{Z} \leq & c T^{\epsilon}\left\{\left\|v_{1}\right\|_{X_a^{s, b}}\left\|u_{1}-u_{2}\right\|_{X^{\kappa, b}}+\left\|u_{2}\right\|_{X_a^{\kappa, b}}\left\|v_{1}-v_{2}\right\|_{X^{s, b}}\right.\\
&+\left(\left\|u_{1}\right\|_{X^{\kappa, b}}+\left\|u_{2}\right\|_{X^{\kappa, b}}\right)\left\|u_{1}-u_{2}\right\|_{X^{\kappa, b}}\}.
\end{aligned}
\end{equation}

Set the ball of $Z:$
$$
B=\left\{(u, v) \in Z ;\|u\|_{Z_{1}} \leq M_{1},\|v\|_{Z_{2}} \leq M_{2}\right\}
$$
where $M_{1}=2 c\left(\left\|u_{0}\right\|_{H^{s}\left(\mathbb{R}^{+}\right)}+\|f\|_{H^{\frac{2 s+1}{4}}\left(\mathbb{R}^{+}\right)}\right)$ and $M_{2}=2 c\left(\left\|v_{0}\right\|_{H^{\kappa}\left(\mathbb{R}^{+}\right)}+\right.$
$\left.\|g\|_{H^{\frac{\kappa+1}{3}}\left(\mathbb{R}^{+}\right)}\right)$.

Restricting $(u, v)$ on the ball $B,$ we have from \eqref{c1} and \eqref{c2}
and  choosing $T=T\left(M_{1}, M_{2}\right)$ small enough, we get
$$
\left\|\Lambda_{1}(u, v)\right\|_{Z_{1}} \leq M_{1}, \quad\left\|\Lambda_{2}(u, v)\right\|_{Z_{2}} \leq M_{2}
$$
and
$$
\left\|\Lambda\left(u_{1}, v_{1}\right)-\Lambda\left(u_{2}, v_{2}\right)\right\|_{Z} \leq \frac{1}{2}\left\|\left(u_{1}, v_{1}\right)-\left(u_{2}, v_{2}\right)\right\|_{Z}
$$

Thus $\Lambda$ defines a contraction map in $Z \cap B$ and we obtain a fixed point in $(u, v)$ in
$B .$ Therefore,
$$
(u, v):=\left(\left.u\right|_{(x, t) \in \mathbb{R}+\times(0, T)},\left.v\right|_{(x, t) \in \mathbb{R}^{+} \times(0, T)}\right).
$$
solves the IBVP \eqref{system} in the sense of distributions.

\section{Proof of Theorem Proof of Theorem \ref{cor} }\label{7}

By using a regularization argument as done in the appendix of the work \cite{CC} the identity \eqref{mass} does work for the solution $(u,v)\in C([0,T^*]:L^2(\R^+)\times L^2(\R^+))$ associated to the initial data $u_0\in L^2(\R^+)$ and $v_0\in L^2(\R^+)$ and the homogeneous boundary data $(f=g\equiv0)$. Then we have that 
$$\|u(t)\|^2_{L^2_x(\R^+)}+\|v(t)\|^2_{L^2_x(\R^+)}=\|u_0\|^2_{L^2_x}+\|v_0\|^2_{L^2_x}.$$

Then we can extended the solution for any time $T>0$.

\section*{Appendix: Proof of Lemmma \ref{lw1}}

Initially we assume $0<a<\frac12$:	

\textbf{Sub-case $\kappa\geq0$:} By using Lemma \ref{l1} it sufficies to consider the case $
|\tau|>10|\xi|^{2},$ which implies that $\left\langle\tau+\xi^{2}\right\rangle \sim\langle\tau\rangle
$. Thus, arguing as in the proof of Lemma \ref{l1} we need to show that the function
\begin{equation}
J(\xi,\tau)= \dfrac{\chi_{\left\{|\tau|>10|\xi|^{2}\right\}}}{\langle\tau+\xi^2\rangle^{2d-\kappa}}\displaystyle\int_{\mathbb{R}}\dfrac{d\xi_2}{\langle \xi_1\rangle^{2\kappa} \langle \xi_2\rangle^{2s}\langle\tau-(a-1)\xi_2^2-2\xi\xi_2+\xi^2\rangle^{{4b-1}}}.
\end{equation} 

If $0\leq \kappa\leq 2d$, we control $J(\xi,\tau)$ by
\begin{equation}
\int_{\mathbb{R}}\dfrac{d\xi_2}{ \langle \xi_2\rangle^{2s}\langle\tau-(a-1)\xi_2^2-2\xi\xi_2+\xi^2\rangle^{{4b-1}}}.
\end{equation}
We split this integral in two regions: $\xi_2\leq 1$ and  $\xi_2>1$. The first one is   easily bounded. The second one is controlled by 
\begin{equation}
c\int_{|\xi_2|>1}\dfrac{d\xi_2}{ | \xi_2|^{2s}\langle\tau-(a-1)\xi_2^{2}-2\xi\xi_2+\xi^2\rangle^{{4b-1}}}.
\end{equation}
This integral is controlled by 

\begin{equation}
c\int_{|\xi_2|>1}\dfrac{d\xi_2}{ \left\langle \frac{1}{|\xi_2|^{-\frac{2s}{4b-1}}}(\tau-(a-1)\xi_2^{2}-2\xi\xi_2+\xi^2)\right\rangle^{{4b-1}}}.
\end{equation} 
Now, this integral is bounded since $2(4b-1)+2s>1$. Thus this integral is controlled for 
$s\geq -\frac12.$ and a adequately $b=b(s)<\frac12$.

\textbf{$\kappa\leq-\frac12$:} In this situation we can assume $|\xi|^2\geq  10 |\tau|$, then $\langle \tau-\xi^2\rangle\sim \langle \xi \rangle^2$. Thus we need to control the following functions
\begin{equation}
J_1(\xi,\tau)=\frac{\langle\tau\rangle^{ \kappa}}{\langle\omega\rangle^{2 d}} \int_{\mathbb{R}^{2}} \frac{\left\langle\xi_{1}\right\rangle^{-2 \kappa}\left\langle\xi_{2}\right\rangle^{-2 s} \chi_{\mathcal{R}_{1}}}{\left\langle\omega_{1}\right\rangle^{2 b}\left\langle\tau_{2}+a \xi_{2}^{2}\right\rangle^{2 b}} d \xi_{2} d \tau_{2},
\end{equation}
\begin{equation}
J_2(\xi_1,\tau_1)=\frac{\left\langle\xi_{2}\right\rangle^{2 s}}{\left\langle\omega_{2}\right\rangle^{2 b}} \int_{\mathbb{R}^{2}} \frac{\left\langle\xi_{1}\right\rangle^{-2 \kappa}\langle\tau\rangle^{ \kappa} \chi_{\mathcal{R}_{2}}}{\left\langle\omega_{1}\right\rangle^{2 b}\langle\omega\rangle^{2 d}} d \xi d \tau
\end{equation}
\begin{equation}
J_3(\xi_1,\tau_1)= \frac{\left\langle\xi_{1}\right\rangle^{-2 \kappa}}{\left\langle\tau_1-\xi_1^2\right\rangle^{2 b}} \int_{\mathbb{R}^{2}} \frac{\langle\tau\rangle^{ \kappa}\left\langle\xi_{2}\right\rangle^{-2 s} \chi_{\mathcal{R}_{3}}}{\langle\xi\rangle^{4 d}\left\langle\tau_2+a\xi_2^2\right\rangle^{2 b}} d \xi_{2} d \tau_{2}.
\end{equation} 
Where
\begin{equation}
\mathcal{R}_1= \bigg\lbrace |\xi_2|\geq 1, |\omega|=\max \{|\omega|,|\omega_1|, |\omega_2|\}\bigg\rbrace\cup\bigg\{|\xi_2| \leq 1\bigg\}\subset \mathbb{R}^4_{\xi,\tau,\xi_2,\tau_2};
\end{equation}

\begin{equation}
\mathcal{R}_2= \bigg\{ |\xi_2|\geq 1, |\omega_1|=\max \{|\omega|,|\omega_1|, |\omega_2|\}\bigg\}\subset \mathbb{R}^4_{\xi,\tau,\xi_2,\tau_2}
\end{equation}
and
\begin{equation}
\mathcal{R}_3= \bigg\{ |\xi_2|\geq 1, |\tau_2+a\xi^2|=\max \{|\omega|,|\omega_1|, |\omega_2|\}\bigg\}\subset \mathbb{R}^4_{\xi,\tau,\xi_2,\tau_2}.
\end{equation}

It follows that
$$3\max \{|\omega|,|\omega_1|, |\omega_2|\}\geq (1-2a)\max\{\xi^2,\xi_1^2\}\geq \frac{1-2a}{4} \xi_2^2.$$

By using that $w:=\tau-\xi^2\cong|\xi|^2$, we control $J_1$ by
\begin{equation}
\frac{\langle\tau\rangle^{ \kappa}}{\langle\xi\rangle^{4 d}} \int_{\mathbb{R}^{2}} \frac{\left\langle\xi_{1}\right\rangle^{-2 \kappa}\left\langle\xi_{2}\right\rangle^{-2 s} \chi_{\mathcal{R}_{1}}}{\left\langle\omega_{1}\right\rangle^{2 b}\left\langle\tau_{2}+a \xi_{2}^{2}\right\rangle^{2 b}} d \xi_{2} d \tau_{2}.
\end{equation} 
It follows that this integral is controlled as in the estimate of \eqref{j1} in the proof of Lemma \ref{l1}.

Now we bound $J_2$ by
\begin{equation}\label{03121}
\frac{\left\langle\xi_{2}\right\rangle^{2 s}}{\left\langle\omega_{2}\right\rangle^{2 b}} \int_{\mathbb{R}^{2}} \frac{\left\langle\xi_{1}\right\rangle^{-2 \kappa}\langle\tau\rangle^{ \kappa} \chi_{\mathcal{R}_{2}}}{\left\langle\omega_{1}\right\rangle^{2 b}\langle\xi\rangle^{4 d}} d \xi d \tau.
\end{equation} 
Then using the definition of $\mathcal{R}_{2}$ we have that
\begin{equation}\label{03122}\frac{\langle\xi_2\rangle^{2s}\langle\xi_1\rangle^{-2\kappa}}{\langle\xi\rangle^{4d} \langle w_2\rangle^{2b}}\leq c \frac{\langle\xi_2\rangle^{2s}\langle\xi_1\rangle^{-2\kappa}}{\langle\xi\rangle^{4d}\langle \xi_2\rangle^{4b}}=c \frac{\langle\xi_2\rangle^{2s}\langle\xi\rangle^{2k}\langle\xi_1\rangle^{-2\kappa}}{\langle\xi\rangle^{4d+2k}\langle \xi_2\rangle^{4b}}\leq c  \frac{\langle\xi_2\rangle^{2s-4b+2|\kappa|}}{\langle\xi\rangle^{4d+2k}}. \end{equation}

Then if $s-|\kappa|\leq 4b$ we control \eqref{03121} by
\begin{equation}\label{03123}
\int_{\mathbb{R}^{2}} \frac{\langle\tau\rangle^{ \kappa} \chi_{\mathcal{R}_{2}}}{\left\langle\omega_{1}\right\rangle^{2 b}\langle\xi\rangle^{4 d+k}} d \xi d \tau.
\end{equation}
Now, by assuming $2b-\kappa>1$ we use Lemma \ref{lemagtv} to bound \eqref{03123} by
\begin{equation}\label{03124}
\int_{\mathbb{R}^{2}} \frac{c}{\left\langle\omega_{1}-\tau\right\rangle^{2 b-\kappa-1}\langle\xi\rangle^{4 d+k}} d \xi .
\end{equation}
Finally, this last integral is finite if $4d+\kappa >1$. The estimate of $J_3$ follows the same ideas of the  $J_2$ estimate. 

The case $a>\frac12$ follows the same ideas of the case $a<\frac12$ with the decomposition of the proof of Lemma \ref{l2}. Finally, the case $a=\frac12$ is similar of Lemma \ref{l1} case $a=\frac12$, since the relation dispersion it is not need to treat this estimates, since we have assume that $\kappa\leq 2d$ and $s,\ \kappa\geq0$.


	
	\bibliographystyle{abbrv}
	\addcontentsline{toc}{section}{References}

\begin{thebibliography}{10}
		
		\bibitem{pava-2007}
		J.~Angulo and F.~Linares.
		\newblock {\em Periodic pulses of coupled nonlinear {S}chr{\"o}dinger equations in
		optics.}
		\newblock { Indiana University Mathematics Journal}, 56(2):847--878, 2007.
		
	
		\bibitem{barbosa2018}
		I. ~Barbosa.
		\newblock \textit{The Cauchy problem for nonlinear quadratic interactions of the Schr{\"o}dinger type in one dimensional space.}
		\newblock {Journal of Mathematical Physics} 59(7): 2018
	
	
	

		
		\bibitem{bourgain-1993}
		J.~Bourgain.
		\newblock \textit{Fourier transform restriction phenomena for certain lattice subsets
		and applications to nonlinear evolution equations {I}, {II}.}
		\newblock { Geometric and Functional Analysis}, 3(3):209--262, 1993.
		
			\bibitem{butcher1990}
		P.~Butcher and D.~Cotter
		\newblock \textit{The elements of nonlinear optics,}
		\newblock{ Cambridge university press}, 1990
		
		\bibitem{CHEN2020}	J. Chen, J. Ge, D. Lu and W. Hu, {\it A simple approach to study the boundary-induced trajectory evolution of spatial nonlocal quadratic solitons: Based on the Green’s function method}, Applied Mathematics Letters 102 (106--108) (2020)
	
		
	
		
		
		
	\bibitem{cavalcante}	M. Cavalcante, {\it The initial-boundary value problem for some quadratic nonlinear Schr\"odinger equations on the half-line}, Differential Integral Equations 30 (7–8) (2017)
		521–554.
		
		
	\bibitem{Cav1}	M. Cavalcante,{\it  Initial boundary value problems for some nonlinear dispersive models on the half-line: a review and open problems}, S\~ao Paulo Journal of Mathematical Science, special section: Nonlinear dispersive equations , São Paulo J. Math. Sci. 13, 418–434 (2019)
		
		
		
	\bibitem{CC} M. Cavalcante and A.J. Corcho, {\it Well-posedness and lower bounds of the growth of weighted norms
		for the Schrödinger–Korteweg–de Vries interactions on the
		half-line.} J. Evol. Equ. https://doi.org/10.1007/s00028-020-00566-1. 2020

		
		\bibitem{colin2009stability}
		M.~Colin, T.~Colin, and M.~Ohta.
		\newblock \textit{Stability of solitary waves for a system of nonlinear Schr{\"o}dinger
		equations with three wave interaction.}
		\newblock { Annales de l'Institut Henri Poincare (C) Non Linear Analysis},
		26(6):2211--2226, 2009.
		
		
		
			
		\bibitem{CK} J. E. Colliander,  C. E. Kenig,: {\it The generalized Korteweg-de Vries equation on the half
			line}, Comm. Partial Differential Equations, 27 (2002), no. 11/12, 2187--2266.
		
		

		
		
	\bibitem{CO} A. J. Corcho, S. Correia, F. Oliveira and  J. Drumond Silva. {\it On a nonlinear Schrödinger system arising in quadratic media.}  Communications in Mathematical Sciences - 17(4) (2019),  969--987.
		
	
		
		\bibitem{ginibre-1997}
		J.~Ginibre, Y.~Tsutsumi, and G.~Velo.
		\newblock \textit{On the {C}auchy problem for the {Z}akharov system.}
		\newblock { Journal of Functional Analysis}, 151(2):384--436, 1997.
		
		\bibitem{hayashi2013}
		N.~Hayashi, T.~Ozawa, and K.~Tanaka.
		\newblock \textit{On a system of nonlinear Schr{\"o}dinger equations with quadratic
		interaction.}
		\newblock { Annales de l'Institut Henri Poincare (C) Non Linear Analysis},
		30(4):661--690, 2013.
		
	\bibitem{Hayashi 2}N. Hayashi, C. Li and T. Ozawa, {\it Small data scattering for a system of nonlinear Schr\"odinger equations}, Differ. Equ. Appl. 3, (2011), 415--426.
		
		
	\bibitem{Hoshino}G. Hoshino and T. Ozawa, {\it Analytic smoothing effect for a system of nonlinear Schr\"odinger equations}, Differ. Equ. Appl. 5, (2013), 395--408.
		
	\bibitem{Holmer} J. Holmer; \textit{The initial-boundary value problem for the Korteweg-de Vries equation.} Communications in Partial Differential Equations, 31 (2006), 1151--1190.
		
		
		\bibitem{karamzin-1974}
		Y.~N. Karamzin and A.~Sukhorukov.
		\newblock \textit{Nonlinear interaction of diffracted light beams in a medium with
		quadratic nonlinearity: mutual focusing of beams and limitation on the
		efficiency of optical frequency converters.}
		\newblock { JETP Lett}, 20(11):339--343, 1974.
		
		
		\bibitem{lopes2005}O.  Lopes
			{\it Stability of solitary waves of some coupled systems}, {Nonlinearity}, 19 (1),
			2005
			
		
		
		
		\bibitem{KPV} C. E. Kenig, Gustavo Ponce, and Luis Vega, {\it Oscillatory integrals and regularity of dispersive equations}, Indiana Univ. Math. J., Vol.40 (1991), 33-69.
		
		
		
		
		
		\bibitem{li2014recent}
		C.~Li and N.~Hayashi.
		\newblock \textit{Recent progress on nonlinear Schr{\"o}dinger systems with quadratic
		interactions.}
		\newblock { The Scientific World Journal}, 2014, 2014.
		
	\bibitem{linares}{	F. Linares and G. Ponce,}{\it  Introduction to Nonlinear Dispersive Equations,} 2nd edition
		(Springer, 2014).
		
		\bibitem{menyuk-1994}
		C.~Menyuk, R.~Schiek, and L.~Torner.
		\newblock \textit{Solitary waves due to $\chi$ (2): $\chi$ (2) cascading.}
		\newblock { JOSA B}, 11(12):2434--2443, 1994.
		
		
		
		\bibitem{numerical}
		C. F. de Oliveira and et al.
		\newblock Numerical stability of solitons waves through splices in quadratic optical media.
		\newblock { Acta Scientiarum. Technology}, 42:e46881--e46881, 2020.
			
			
			
		\bibitem{Pastor1} N. Noguera and A. Pastor, {\it On the dynamics of a quadratic Schr\"odinger system in dimension $n=5$.} Dynamics of Partial Differential Equations, v. 17, p. 1-17, 2020
			
			
			\bibitem{Pastor2} N. Noguera and A. Pastor. {\it Scattering of radial solutions for quadratic-type Schr\"{o} dinger systems in dimension five.} arXiv preprint arXiv:2008.12696 (2020).
			
	
		
	\bibitem{tao}	{T. Tao,} {\it Nonlinear Dispersive Equations: Local and Global Analysis,} Vol. 106 (American Mathematical Society, Providence, RI, 2006).
		
		\bibitem{yew-2000}
		A.~Yew.
		\newblock \textit{Stability analysis of multipulses in nonlinearly-coupled
		{S}chr{\"o}dinger equations.}
		\newblock { Indiana University Mathematics Journal}, 49(3):1079--1124, 2000.
		
		
	
		
		
	\end{thebibliography}
\end{document}